\documentclass{amsart}
\usepackage{amssymb}
\usepackage{latexsym}
\usepackage{euscript}
\usepackage{datetime}

\usepackage[]{graphicx}
\usepackage{amssymb}
\usepackage[usenames,dvipsnames,svgnames,table]{xcolor}

\def\sD{{\mathfrak D}}   \def\sE{{\mathfrak E}}   
   \def\sH{{\mathfrak H}}   
   \def\sK{{\mathfrak K}}   \def\sL{{\mathfrak L}}
\def\sM{{\mathfrak M}}      
      \def\sR{{\mathfrak R}}
      
      \def\sX{{\mathfrak X}}
\def\sY{{\mathfrak Y}}

\def\bB{{\mathbf{B}}}

\def\wt#1{{{\widetilde #1} }}

\def\bm\chi{\mbox{\boldmath$\chi$}}
\def\half{{\frac{1}{2}}}

\def\ker{{\rm ker\,}}
\def\ran{{\rm ran\,}}
\def\cran{{\rm \overline{ran}\,}}
\def\dom{{\rm dom\,}}
\def\mul{{\rm mul\,}}

\def\cdom{{\rm \overline{dom}\,}}
\def\clos{{\rm clos\,}}

\let\xker=\ker \def\ker{{\xker\,}}

\DeclareMathOperator{\hoplus}{\, \widehat \oplus \,}

\newtheorem{theorem}{Theorem}[section]

\newtheorem{corollary}[theorem]{Corollary}
\newtheorem{lemma}[theorem]{Lemma}

\theoremstyle{definition}

\newtheorem{remark}[theorem]{Remark}
\newtheorem{definition}[theorem]{Definition}

\numberwithin{equation}{section}

\begin{document}

\title[Lebesgue type decompositions and Radon-Nikodym derivatives]
{Lebesgue type decompositions and Radon-Nikodym derivatives for
pairs of bounded linear operators}
\author{S.~Hassi}
\author{H.S.V.~de~Snoo}

\address{Department of Mathematics and Statistics \\
University of Vaasa \\
P.O. Box 700, 65101 Vaasa \\
Finland}
\email{sha@uwasa.fi}

\address{Bernoulli Institute for Mathematics, Computer Science and Artificial Intelligence \\
 University of Groningen \\
P.O. Box 407, 9700 AK Groningen \\
Nederland}
\email{hsvdesnoo@gmail.com}





\begin{abstract}
For a pair of bounded linear Hilbert space operators $A$ and $B$
one considers the
Lebesgue type decompositions of $B$ with respect to $A$
into an almost dominated part and a singular part,
analogous to the Lebesgue decomposition
for a pair of measures (in which case one speaks of an absolutely
continuous and a singular part).
A complete parametrization of all Lebesgue type decompositions will be given,
and the uniqueness of such decompositions will
be characterized.
In addition, it will be shown that the almost dominated part of $B$
in a Lebesgue type decomposition
has an abstract Radon-Nikodym derivative with respect to the operator $A$.
 \end{abstract}

\maketitle

\section{Introduction}

Let $\sE$, $\sH$, and $\sK$ be Hilbert spaces and
$A \in \bB(\sE,\sH)$ and $B \in \bB(\sE,\sK)$ be bounded
linear operators.
In the present paper it will be shown that there are so-called \textit{Lebesgue type
decompositions} of  the
operator $B \in \bB(\sE,\sK)$ relative to the operator $A \in \bB(\sE,\sH)$
of the form
 \begin{equation}\label{Bpair1}
 B=B_1+B_2, \quad B_1, B_2 \in \bB(\sE,\sK),
\end{equation}
where $\ran B_1 \perp \ran B_{2}$, $B_{1}$ is almost dominated by $A$,
and $B_2$ is singular with respect to $A$; the terminology will be explained below.
The collection of all Lebesgue type decompositions will be parametrized
and a criterion for the uniqueness of such decompositions will be established.
Furthermore, it will be shown that if $B$ has the above Lebesgue type
decomposition \eqref{Bpair1} with respect to $A$,
then there exists a uniquely determined closed linear operator $C$
from $\sH$ to $\sK$ satisfying a certain minimality condition, that, in general,
is unbounded, such that
 \begin{equation}\label{Bpair2}
 B_1=CA;
\end{equation}
for details, see Theorem \ref{thmRNder} below.
 This operator $C$ will be called the \textit{Radon-Nikodym derivative}
 of $B_1$ with respect to $A$.
The above results are the abstract analogs of the usual Lebesgue decomposition
of a   finite measure into an absolutely continuous part and a singular part,
and of the corresponding Radon-Nikodym derivative for the absolutely continuous part.
There are, indeed, situations in measure theory where actually
there is more than one Lebesgue type decomposition.

The above results can be interpreted as special cases
of corresponding results for linear relations.
For any linear relation  $T$ from a Hilbert space  $\sH$ to a Hilbert space $\sK$
there exist the so-called Lebesgue type decompositions of the form
\[
 T=T_1 + T_2
\]
such that $\ran T_{1} \perp \ran T_{2}$,  $T_1$ is a regular operator
(an operator whose closure in $\sH \times \sK$ is an operator, i.e., a closable operator),
and $T_2$ is a singular relation (a relation whose closure
is the product of closed linear subspaces in $\sH$ and $\sK$, respectively).
 A general treatment of Lebesgue type decompositions
 of a linear relation $T$ was carried out in the recent paper \cite{HSS2018}.
For instance, in that paper an explicit parametrization of all
Lebesgue type decompositions of $T$ was established
and the case where the Lebesgue type decomposition
of $T$ is unique has been characterized therein.
In the setting of a pair of positive operators this type
of uniqueness result goes back to Ando \cite{An}.

In the present paper Lebesgue type decompositions for
linear relations which are simultaneously operator ranges are studied (see \cite{FW});
in what follows such   linear relations are called shortly   \textit{operator range relations}.
Operator range relations coincide with the linear relations of the form
\begin{equation}\label{Tpair0}
 L(A,B)=\{\,\{Af,Bf\}:\, f\in\sE\,\},
\end{equation}
where $A \in \bB(\sE,\sH)$ and $B \in \bB(\sE,\sK)$.
The Lebesgue type decompositions of the relation $L(A,B)$
correspond to the Lebesgue type decompositions
of the operator $B$ with respect to the operator $A$ in \eqref{Bpair1}.
In particular, $B$ is almost dominated by $A$
precisely if the corresponding relation $L(A,B)$ is regular, i.e.,
$L(A,B)$ is a closable operator;
moreover, $B$ is singular with respect to $A$ precisely if $L(A,B)$ is singular.
If $L(A,B)$ is regular,
then its closure is the Radon-Nikodym derivative of $B$
with respect to $A$ mentioned in \eqref{Bpair2};
the precise meaning of this statement will be explained later.
The approach in the present paper makes it possible
to further develop the results established
in \cite{HSS2018}, including uniqueness results, to the
setting of operator range relations;
see also \cite{HSeSnSz,HSeSnLeb2009,HSnSz},

 The topic of the present paper was inspired by the work of Ando
about Lebesgue type decompositions for pairs of bounded nonnegative
operators \cite{An},  and the work of Simon
about Lebesgue decompositions for nonnegative forms \cite{S3},
see also \cite{HSeSnLeb2009} and \cite{HSeSnSz},
and Kosaki's work on the Radon-Nikodym derivative in the setting of $C^*$-algebras  \cite{Kos}.
The context of  pairs of bounded linear operators which are not necessarily nonnegative,
 is connected with the work of Mac~Nerney, Kaufman,
and others; see \cite{Kau78, Kau79, Kau83, Kau84, Koli14}.
Moreover, there are strong  connections
with the work of Izumino \cite{I89a, I89b, I93}
and of Izumino and Hirasawa \cite{IH}.  \
They treated the case where $L(A,B)$ in \eqref{Tpair0}  is a densely
defined operator.
 The parametrization of the Lebesgue type decompositions
 and the notion of Radon-Nikodym seem to
be new even in the case where $L(A,B)$ is densely defined.
 Moreover, at this point, it should be mentioned that, although
 the paper is inspired by \cite{An}, \cite{S3}, and \cite{Kos},
the decompositions \eqref{Bpair1} here are concerned
with pairs $A$ and $B$, where $B_1$ is almost dominated by $A$
and $B_2$ is singular with respect to $A$. However,
the decompositions of Ando and Simon, and Kosaki's
Radon-Nikodym derivatives belong to a slightly different setting;
this setting will be considered in further work.

Here is a brief description of the contents of the paper.
The concept of operator ranges, including their natural topologies,
will be reviewed in Section \ref{sec3}.
Operator range relations and a special normalized class
of them are treated in Section \ref{sec4}.
As an application, a construction of the operator range representation
of a closed relation is derived in Section \ref{sec5}.
This construction resembles a measure theoretic treatment
of Radon-Nikodym derivatives for a pair of positive measures,
and, as a bonus, leads to a natural introduction
of Radon-Nikodym derivatives for pairs of bounded linear operators
in Section \ref{sec8}.
As a preparation for Lebesgue type decompositions for pairs of bounded operators,
some characterizations of regular and singular operator range relations
along the lines of \cite{HSS2018} are given in Section \ref{sec6}.
The corresponding classification for pairs of
bounded linear operators can be found in Section \ref{sec7}.
This involves the notions of domination and almost domination
of a bounded operator with respect to another bounded operator,
which correspond to the concept of absolute continuity
of a positive measure with respect to another positive measure.
Similarly the notion of singularity of a pair of bounded operators
is defined as an operator analog
for the concept of singularity of a pair of positive measures.
In Section \ref{sec8} the abstract Radon-Nikodym derivative is introduced and investigated.
The definition of Radon-Nikodym derivative given here involves an optimality property,
see Lemma~\ref{thmRNder0} and Theorem~\ref{thmRNder} and, in fact, this notion is uniquely determined
in the general setting of operator range relations.
 Furthermore, it is shown that the Radon-Nikodym derivatives for operator range relations
admit similar properties known to hold for pairs of positive measures; see Theorem \ref{rrnn1}.
Finally, all Lebesgue type decompositions for pairs of bounded linear operators are
described in Section \ref{sec9} with a uniqueness result analogous to that of Ando \cite{An}.

Further work will be concerned with the situation that the operators  $A$ and $B$ are
nonnegative; such cases have been considered by Ando, Simon, Kosaki, and by later authors.
In particular, there will be an explicit connection to the recent papers
by Z.~Sebesty\'{e}n, Zs.~Tarcsay, and T.~Titkos; see for instance \cite{STT1, STT2, Tar16}.
Moreover, further work will also be connected with the situation where at least one of the operators $A$  and $B$
is not bounded.

\section{Linear relations admitting an operator range representation}\label{sec3}

In this section the special class of linear relations which, in addition,
are operator ranges, will be introduced.
In what follows, such linear relations are briefly called \textit{operator range relations.}
This class extends the class of closed linear relations and the operator range relations
have a number of useful properties.
The introduction will be facilitated by a brief treatment of operator ranges
in Hilbert spaces. The notions of operator ranges and operator range relations in the present sense
go back to \cite{FW},  \cite{Kau79}, \cite{La1980}, \cite{Labro08}, \cite{Mac59}, \cite{Mac70}.
A brief survey is given in this section.
For the convenience of the reader, some proofs are included.

\begin{definition}\label{opran2}
Let $\sX$ be a Hilbert space with inner product $(\cdot, \cdot)$.
A subspace $\sM$ of   $\sX$, together with an inner-product
$(\cdot,\cdot)_+$ on $\sM$, is said to be an operator range in $\sX$, if
 \begin{enumerate}
\def\labelenumi {\rm (\alph{enumi})}
\item $\sM$ is a Hilbert space when equipped with
$(\cdot,\cdot)_+$;

\item $\|u\|_+ \ge c \|u\|_\sH$, $u \in \sM$, for some $c > 0$.
\end{enumerate}
In particular, a closed linear subspace is an operator range.
\end{definition}

The terminology operator range is motivated the following lemma.
If $\sM \subset \sX$ is the range of a bounded operator,
then there is a natural inner-product on $\sM$,
that makes $\sM$ an operator range.

\begin{lemma}\label{opran00}
Let $\sX$ and $\sY$ be Hilbert spaces and let $Z \in \mathbf{B}(\sY, \sX)$.
Then the linear space $\sM=\ran Z$, equipped with the inner-product
\begin{equation}\label{opran+}
 (Zx, Zy)_+=(x,y), \quad x,y \in \sY \ominus \ker Z,
\end{equation}
is an operator range.
\end{lemma}

\begin{proof}
Assume that $Z$ is not the zero operator.
 Note that  $\sM=\ran Z$ with $(\cdot,\cdot)_+$
in \eqref{opran+} is indeed an inner-product space.
To see that it is complete, let $(Zx_n)$ be a Cauchy sequence in $(\sM, (\cdot,\cdot)_+)$
with $x_n \in \sY \ominus \ker Z$.
Then $(x_n)$ is a Cauchy sequence in $\sY \ominus \ker Z$.
Thus $x_n \to x$ for some $x \in \sY \ominus \ker Z$ and $Zx_n \to Zx$,
since $Z \in \mathbf{B}(\sY, \sX)$.
This shows (a) in Definition \ref{opran2}.
By definition $\|Zx\|_+=\|x\|$, $x \in \sY \ominus \ker Z$,
and $\|Zx\| \leq \|Z\| \|x\|$, $x \in \sY$,  lead to the inequality
\[
  \|Zx\|_+ =\|x\| \ge \frac{1}{\|Z\|} \|Zx\|_\sX, \quad x \in \sY \ominus \ker Z,
\]
which gives (b) in Definition \ref{opran2}.
\end{proof}

There is also a converse to Lemma \ref{opran00}.

\begin{lemma}\label{opran0}
Let $(\sM, (\cdot,\cdot)_+)$ be an operator range in $\sX$.
Then there exists an operator $Z \in \mathbf{B}(\sX)$
such that $\sM=\ran Z$ and
\begin{equation}\label{opran+-}
 (Zx, Zy)_+=(x,y), \quad x,y \in \sX \ominus \ker Z.
\end{equation}
The operator $Z$ may be chosen to be nonnegative.
 \end{lemma}

\begin{proof}
 Let $\imath : \sM \to \sX$ be the identification map, where each space has its own topology.
Then it follows from (b) that $c \,\|\imath \,x\| \leq \|x\|_+$, so that  $\imath$ is bounded.
Its adjoint $\imath^{\times}$ from $\sX$ to $\sM$
is a bounded mapping and one has the polar decomposition
(cf. \cite[Section~VI~2.7]{Kato})
\[
 \imath^{\times}=Z \,|\,\imath^{\times}|,
\]
where $Z: \sX \to \sM$
is the unique partial isometry with initial space $\cran |\,\imath^\times|$
and final space $\cran \imath^\times$. Since the last space is the orthogonal complement
in $\sM$ of $\ker \imath$, one sees that $\cran \imath^\times=\sM$. Consequently, $\ran Z=\sM$,
and \eqref{opran+-} holds as $Z$ is a partial isometry from $\sX$ to $\sM$.
 Finally, it remains to observe that
\[
c \|Z x\| \leq \|Zx\|_+ \leq \|x\|, \quad x \in \sX,
\]
thanks to (b) and the fact that $Z$ is a partial isometry. Thus $Z \in \mathbf{B}(\sX)$.

For the proof of the last statement, consider $Z \in \bB(\sX)$ and its polar decomposition $Z=|Z^*|C$, where
$C \in \bB(\sX)$ is the unique partial isometry with initial space  $\cran Z^*$ and final space $\cran |Z^*|$.
Observe that, by the Douglas Lemma,  $\ran |Z^*|=\ran Z$, so that $\ran |Z^*| =\sM$  and, clearly,
\[
 (|Z^*|Cx, |Z^*|Cy)_+=(x,y)=(Cx,Cy), \quad x,y \in (\ker Z)^\perp.
\]
This implies that
\[
 (|Z^*|u, |Z^*|v)_+=(u,v), \quad u,v \in (\ker |Z^*|)^\perp,
\]
 where $|Z^*| \in \mathbf{B}(\sX)$ is nonnegative.
 \end{proof}

 According to Lemma \ref{opran00} and  Lemma \ref{opran0},
an operator range $\sM$
is naturally parametrized by a bounded operator $Z \in \mathbf{B}(\sY, \sX)$.
The subspace $\ker Z \subset \sY$ is called the \textit{redundant part}
of this parametrization, as it does not contribute to $\ran Z$.
The restriction $Z_0$ of $Z$ to $\sY \ominus \ker Z$ is called the
\textit{reduced part} or \textit{reduction} of $Z$:
\begin{equation}\label{redux}
Z=(Z_0 \,;\, O_{\ker Z}).
\end{equation}
 The topology of $\sM$ is uniquely determined by the reduced part of the
representing operator in a sense to b explained below.

\begin{lemma}\label{opran+++}
Let $\sM$ be an operator range in $\sX$. Assume that
there exist Hilbert spaces $\sY$ and $\sY_1$ and operators
$Z \in \mathbf{B}(\sY, \sX)$ and $Z_1 \in \mathbf{B}(\sY_1, \sX)$,
which each satisfy the conditions in  Lemma {\rm \ref{opran00}}.
Then there is a bounded and boundedly invertible  operator
$W \in \mathbf{B}(\sY_1,\sY)$ such that
\begin{equation}\label{opran++}
Z_1=ZW, \quad \ker W=\ker Z_1, \quad \ran W=\cran Z^*.
\end{equation}
Consequently, the topologies induced on $\sM$ by $Z$
and by $Z_1$ are equivalent.
\end{lemma}

\begin{proof}
If $\sM=\ran Z_1$ with $Z_1 \in \mathbf{B}(\sY_1,\sX)$
with a Hilbert space $\sY_1$, then
it is clear that the operators $Z_1$ and $Z$ have the same range.
Hence, by the Douglas lemma,
 there is an  operator
$W \in \mathbf{B}(\sY_1,\sY)$ such that  \eqref{opran++} holds.
Note that the operator $W$ in Lemma \ref{opran+++} is a  bounded bijective
mapping from $(\ker Z_1)^\perp$ onto $(\ker Z)^\perp$,
as follows from the closed graph theorem.
\end{proof}

\begin{lemma}
Let $\sM=\ran Z$ be an operator range as in  Lemma {\rm \ref{opran+}}
and let $Z_0$ be the reduced part of $Z$.
Then the following statements are equivalent:
\begin{enumerate}\def\labelenumi{\rm(\roman{enumi})}
\item $\sM$ is an operator range in $\sX$ that is closed;
\item $\ran Z$ or, equivalently, $\ran Z_0$ is closed;
\item $Z_0$ is bounded and boundedly invertible.
\end{enumerate}
\end{lemma}

Operator ranges in a Hilbert space $\sX$ extend the notion of closed linear subspaces of $\sX$.
In particular, operator ranges form a lattice; cf. \cite{FW}.

\medskip

The previous notion of operator range will now be extended
to subspaces of product spaces; see for instance \cite{Labro08}.

\begin{definition}
 A linear relation $T$ from $\sH$ to $\sK$ is said to be an \textit{operator range
relation} if its graph is an operator range in $\sH \times \sK$, thus
 $T=\ran \Phi$  for some $\Phi \in \bB(\sE, \sH \times \sK)$,
where $\sE$ is some Hilbert space.
\end{definition}

\begin{lemma}\label{cloclo}
Let $T$ be a closed linear relation from the Hilbert space $\sH$ to the Hilbert space  $\sK$.
Then $T$ is an operator range relation.
\end{lemma}

\begin{proof}
Consider the orthogonal decomposition
$\sE=\sH \times \sK=T \hoplus T^\perp$ and let
$\Phi=P_T$ be the orthogonal projection from
$\sE=\sH\times\sK$ to $T$, so that $T=\ran \Phi$.
\end{proof}

To characterize operator range relations
the following notions and notations are needed.
For a pair of operators $A\in\bB(\sE,\sH)$ and
$B\in\bB(\sE,\sK)$
 the linear relation $L(A,B)$ from  $\sH$ to $\sK$ is defined by
 \begin{equation*}
 L(A,B)=\{\,\{Af,Bf\}:\, f\in\sE\,\}.
\end{equation*}
 Let $P_1$ and $P_2$ be the orthogonal projections from
$\sH\times\sK$ onto the component subspaces
$\sH\times\{0\}$ and $\{0\}\times\sK$, respectively.
The mappings $\iota_{1} \{\varphi,0\}=\varphi$ and
$\iota_{2} \{0,\psi\}=\psi$
identify $\sH\times\{0\}$ with $\sH$ and
$\{0\}\times\sK$ with $\sK$, respectively.
In the following theorem the operator range relations
are characterized; see \cite[Theorem 1.10.1]{BHS}.

\begin{theorem}\label{Tclosed}
 Let $\sH$ and  $\sK$ be Hilbert spaces
 and let $T$ be a linear relation from  $\sH$ to $\sK$.
Then there are the following statements:
\begin{enumerate}\def\labelenumi{\rm(\alph{enumi})}
\item If  $T=\ran \Phi$, where $\Phi\in\bB(\sE, \sH\times\sK)$,
then $T=L(A,B)$ with
\[
  A=\iota_{1} P_1 \Phi \in \bB(\sE,\sH) \quad \mbox{and}
  \quad B=\iota_{2} P_2 \Phi \in \bB(\sE,\sK).
\]
\item If $T=L(A,B)$, where $A\in\bB(\sE,\sH)$ and
$B\in\bB(\sE,\sK)$, then $T=\ran \Phi$ with
\[
 \Phi \in \bB(\sE,\sH\times \sK) \quad \mbox{and}
 \quad \Phi f=\{Af, Bf\}, \quad f \in \sE.
\]
\end{enumerate}
\end{theorem}

\begin{corollary}
If $T$ is an operator range relation from $\sH$ to $\sK$, then $\dom T$ and $\ker T$ are
operator ranges in $\sH$, while $\ran T$ and $\mul T$ are operator ranges in $\sK$.
\end{corollary}

The concept of an operator range relation extends in a certain way
the notion of a closed linear relation.
For instance, sums, and intersections, as well as Cartesian
products of operator range relations are again operator range relations.
In particular, operator range relations form a lattice; cf. \cite{FW}.
Notice also that if the relation $T$ is itself a range of some closed linear relation
$H$ from a Hilbert space $\sE$ to the Hilbert space
$\sH\times \sK$, then $T$ is still an operator range relation.

 \section{Operator range relations and normalized pairs}\label{sec4}

Let $\sE$, $\sH$, and $\sK$ be Hilbert spaces
and consider the pair of operators $A\in\bB(\sE,\sH)$ and
$B\in\bB(\sE,\sK)$. It will be convenient to recall from \cite{HLabS20b, MS08}, see also \cite{FW},
the following auxiliary column and row operators.
The \textit{column operator} $c(A,B)$ from  $\sE$ to $\sH \times \sK$
is defined by
\begin{equation}\label{column}
c(A,B) = \begin{pmatrix} A  \\ B  \end{pmatrix}, \quad \mbox{i.e.,} \quad
 c(A,B) \varphi= \begin{pmatrix} A \varphi \\ B \varphi \end{pmatrix},
\quad \varphi \in \sE.
\end{equation}
 Then $c(A,B)$ belongs to $\bB(\sE,\sH\times\sK)$.
Likewise,  the \textit{row operator}
from $\sH \times \sK$ to $\sE$ is defined by
\begin{equation}\label{row}
   r(A, B)  =  \begin{pmatrix} A & B \end{pmatrix}
  \quad \mbox{i.e.,} \quad
  r(A, B)  \begin{pmatrix} h \\ k \end{pmatrix}=Ah +Bk,
  \quad h \in \sH, \,\,k \in \sK.
\end{equation}
 Then $r(A,B)$ belongs to $\bB(\sH\times\sK, \sE)$ and
\begin{equation}\label{rowa}
 \ran r (A, B)=\ran A+\ran B.
\end{equation}
Moreover, it is clear from \eqref{column} and \eqref{row} that
\begin{equation}\label{rowb}
c(A,B)^*=r (A^*, B^*).
\end{equation}
Therefore, it follows from \eqref{column} and \eqref{row} that
 \begin{equation}\label{column1}
   c(A,B)^*c(A,B)=A^{*}A+B^{*}B,
\end{equation}
and that
  \begin{equation}\label{column2}
   c(A,B)c(A,B)^*=\begin{pmatrix}
                 AA^* & AB^* \\
                 BA^* & BB^*
               \end{pmatrix}.
\end{equation}

\medskip

For a pair of operators $A\in\bB(\sE,\sH)$ and $B\in\bB(\sE,\sK)$
 the linear relation $L(A,B)$ from  $\sH$ to $\sK$ is defined by
 \begin{equation}\label{Tabegin}
 L(A,B)=\{\,\{Af,Bf\}:\, f\in\sE\,\},
\end{equation}
so that $L(A,B)$ is an operator range relation
as in Theorem \ref{Tclosed}, since in the present notation
\begin{equation}\label{Tab**}
L(A,B)=\ran c(A,B).
\end{equation}
The operator range relation $L(A,B)$  is sometimes called a \textit{quotient},
as, indeed, in the sense of the product of
linear relations one can write $L(A,B)$ as $BA^{-1}$; cf. \cite{Kau78}.
 Note that the domain and the range of  $L(A,B)$
in \eqref{Tabegin} are given by
\[
 \dom L(A,B)=\ran A, \quad \ran L(A,B)=\ran B,
\]
while the kernel and the multivalued part of
$L(A,B)$ are given by
\begin{equation}\label{Tabb}
 \ker L(A,B)=A(\ker B), \quad \mul L(A,B)= B(\ker A).
\end{equation}

 Recall that for a linear relation $T$ from $\sH$ to $\sK$ the adjoint $T^*$
is an automatically closed linear relation from $\sK$ to $\sH$ given by
\begin{equation}\label{adjo}
T^*=JT^\perp=(JT)^{\perp},
\end{equation}
where $J$ stands the flip-flop operator $\{f,g\} \mapsto \{g,-f\}$; in other words
\begin{equation}\label{adjoo}
T^*=\{\,\{h,k\} \in \sK \times \sH :\, (g,h)=(f,k) \mbox{ for all } \{f,g\} \in T\,\}.
\end{equation}
With $J$ the definition of the adjoint can be expressed
in geometric terms as follows $T^*=(JT)^\perp$.
In particular, $T^{**}=T^{\perp \perp}$.

Now apply the definition \eqref{adjo} to the linear relation $L(A,B)$ in \eqref{Tabegin}.
Then, by \eqref{Tab**}, the adjoint of $L(A,B)$ is given by $J L(A,B)^{\perp}=J(\ran c(A,B))^{\perp}$,
i.e.
\[
 L(A,B)^{*}=J \,\ker c(A,B)^{*}=J \,\ker r(A^*, B^*).
\]
In other words, the adjoint of $L(A,B)$ is given by
\begin{equation}\label{Tab*}
 L(A,B)^*=\{\,\{k,h\}\in \sK\times\sH:\, B^*k=A^*h\,\},
\end{equation}
cf. \eqref{adjoo}.
At this point, it is useful to introduce the following linear subspaces
$\sD(A,B)$ and $\sR(A,B)$  of $\sK$ and $\sH$ by
\begin{equation}\label{RAB}
\sD(A,B)=\{k \in \sK:\, B^{*}k\in \ran A^{*}\}, \quad
\sR(A,B)=\{h \in \sH:\, A^{*}h\in \ran B^{*}\},
\end{equation}
respectively.  In other words,
$\sD(A,B)$ is the pre-image $(B^*)^{-1}(\ran A^*)$ and
$\sR(A,B)$ is the pre-image $(A^*)^{-1}(\ran B^*)$.
 Note that
\[
\sD(A,B)=\sK \,\, \Leftrightarrow \,\, \ran B^* \subset \ran A^*,
 \quad
\sR(A,B)=\sH \,\, \Leftrightarrow \,\, \ran A^* \subset \ran B^*.
\]
From the expression \eqref{Tab*} it is seen that
the domain and the range of  $L(A,B)^{*}$
are given by the sets in \eqref{RAB}:
\begin{equation}\label{tabst}
 \dom L(A,B)^*=\sD(A,B), \quad  \ran L(A,B)^*=\sR(A,B).
\end{equation}
Moreover, the kernel and multivalued part of $L(A,B)^{*}$ are given by
\begin{equation}\label{tabst+}
 \ker L(A,B)^*=\ker B^*, \quad \mul L(A,B)^*=\ker A^*,
\end{equation}
respectively. Observe that \eqref{tabst} implies that
\begin{equation}\label{tabst1}
\mul L(A,B)^{**}=\sD(A,B)^{\perp}, \quad
\ker L(A,B)^{**}=\sR(A,B)^{\perp}.
\end{equation}

\medskip

It is helpful to state some general properties for the operators $A\in\bB(\sE,\sH)$ and
$B\in\bB(\sE,\sK)$, in which case $c(A,B) \in \bB(\sE,\sH\times\sK)$.
By means of the Douglas lemma, one sees,  for instance, from \eqref{column1} and \eqref{rowb},  that
\begin{equation}\label{column1+}
\begin{split}
 \ran (A^{*}A+B^{*}B)^\half &= \ran (c(A,B)^*c(A,B))^\half = \ran c(A,B)^*\\
 &=\ran r(A^*,B^*)=\ran A^*+\ran B^*,
\end{split}
\end{equation}
cf. \cite{FW}, and   from \eqref{column2} that
\begin{equation}\label{column2+}
   \ran \begin{pmatrix}
                 AA^* & AB^* \\
                 BA^* & BB^*
               \end{pmatrix}^\half
               =\ran (c(A,B)c(A,B)^*)^\half=\ran c(A,B).
\end{equation}
Moreover, it is clear that $\ran c(A,B)$ and $\ran c(A,B)^*$ are simultaneously closed, and
therefore
 the following spaces
\begin{equation}\label{kwat1}
 \ran c(A,B), \quad \ran c(A,B)^*c(A,B), \quad \ran c(A,B)^*, \quad  \ran c(A,B)^*c(A,B),
\end{equation}
 are closed simultaneously;  see e.g. \cite[Theorem~1.3.5]{BHS}.
 The following characterization of closedness is a direct consequence of these considerations.

\begin{lemma}\label{Tclosed0}
Let the linear relation $L(A,B)$ be given by \eqref{Tabegin}.
Then the following statements are equivalent:
\begin{enumerate}\def\labelenumi{\rm(\roman{enumi})}
\item  $L(A,B)$ is closed;
\item $\ran(A^*A+B^{*}B)$ is closed in $\sE$;
\item $ \ran A^*+ \ran B^{*}$ is closed in $\sE$.
 \end{enumerate}
\end{lemma}

\begin{proof}
Thanks to the equivalences in \eqref{kwat1}, the assertions in (i), (ii), and (iii)
follow from   \eqref{Tab**},  \eqref{column1}, and \eqref{rowb} together with \eqref{rowa},
respectively.
 \end{proof}

 The next corollary shows that range space relations admit some specific properties
which are well known in the case of closed operators; cf. the closed graph theorem.
The following result goes back to Foias \cite{FW}, see also \cite{Labro08}; the present
proof seems to be new.

\begin{corollary}
Let $T$ be a range space relation from $\sH$ to $\sK$. Then the following implications hold:
\begin{enumerate}\def\labelenumi{\rm(\roman{enumi})}
\item If $\,\mul T=\{0\}$, then $\dom T$ is closed implies that $T$ is a bounded operator;
\item If $\,\ker T=\{0\}$, then $\ran T$ is closed implies that $T^{-1}$ is a bounded operator.
\end{enumerate}
\end{corollary}

\begin{proof}
(i) By assumption one can write $T=L(A,B)$ with some operators $A$ and $B$ as in \eqref{Tabegin}.
Now the assumption $\mul T=\{0\}$ means that
\[
 \ker A \subset \ker B \quad \mbox{or, equivalently,}
 \quad \cran B^* \subset \cran A^*.
\]
If $\dom T=\ran A$ is closed, then also $\ran A^*$ closed.
Hence, in fact, $\ran B^* \subset\ran A^*$ and
\[
 \ran A^* +\ran B^*=\ran A^* \quad \text{is closed}.
\]
By Lemma \ref{Tclosed0} $L(A,B)$ a closed operator and the statement follows from the closed graph theorem.

(ii) This is obtained by applying (i) to the inverse $T^{-1}$.
\end{proof}

Let $A\in\bB(\sE,\sH)$, $B\in\bB(\sE,\sK)$ and
let the linear relation $L(A,B)$ be given by \eqref{Tabegin}.
In general, there will be some redundance
in the representation \eqref{Tabegin}; cf. \eqref{redux}.
In fact, the redundant part is given by the closed linear subspace
 \[
 \ker c(A,B)=\ker A \cap \ker B.
\]
Observe that it follows from the identities
\[
(\ran (A^*A+B^*B))^\perp=\ker  (A^*A+B^*B) =\ker A \cap \ker B,
\]
that the space $\sE$ has
the following
 orthogonal decomposition:
\begin{equation}\label{Tabbb}
\sE=\cran (A^*A+B^*B)  \oplus (\ker A \cap \ker B).
\end{equation}
 Hence, one may reduce the representation in \eqref{Tabegin}
by introducing  the restrictions
$A_0$ and $B_0$ of $A$ and $B$ to
the closed linear subspace $\sE_0= \cran (A^*A+B^*B)$
of $\sE$, cf. \eqref{redux}, so that
\begin{equation}\label{uno}
A=r (A_0 , O_{\ker A \, \cap \, \ker B}), \quad B=r(B_0, O_{\ker A \,\cap \,\ker B}),
\end{equation}
with respect to the decomposition \eqref{Tabbb}.
It is clear that  $\ker A_0 \cap \ker B_0=\{0\}$, and thus
\begin{equation}\label{Tabegin-}
 L(A,B)=\{\,\{A_0 f,B_0 \, f\}:\, f\in \sE_0\,\}
 \end{equation}
is a representation in terms of the reduced part $c(A_0, B_0)$ of $c(A,B)$.
Recall that reduced representations as in \eqref{Tabegin-}
are uniquely defined, up to everywhere defined operators
which are bounded and boundedly invertible; cf. Lemma \ref{opran+++}.
Furthermore, via \eqref{column2},
 one has from \eqref{uno} that, with respect to the decomposition  \eqref{Tabbb},
\begin{equation}\label{duo}
 A^*A+B^*B=\begin{pmatrix} (A_0)^*A_0+(B_0)^*B_0 & 0 \\0 &0\end{pmatrix}.
\end{equation}
It is clear that $\cran( (A_0)^*A_0+(B_0)^*B_0)=\sE_0$; hence one obtains the following lemma.

\begin{lemma}\label{Tclosed0n}
Let the linear relation $L(A,B)$ be given by \eqref{Tabegin}
and let $A_0$ and $B_0$ as in \eqref{uno}
Then the following statements are equivalent:
\begin{enumerate}\def\labelenumi{\rm(\roman{enumi})}
\item  $L(A,B)$ is closed;
\item $\ran( (A_0)^*A_0+(B_0)^*B_0)=\sE_0$.
 \end{enumerate}
\end{lemma}

 In the rest of this section attention is paid to the case where the linear relation $L(A,B)$ is closed.
In this case the operators $A$ and $B$ representing $L(A,B)$ can be replaced by a normalized pair
$A^*A+B^*B=I$, in which case certain formulas simplify. In fact, it is instructive to first consider the case
where $A^*A+B^*B$ is an orthogonal projection.

\begin{lemma}\label{Tclosed0Q}
Assume that there exists an orthogonal projection $Q$ in $\sE$ such that
\begin{equation}\label{isomTabQ}
  A^*A+B^*B=Q,
\end{equation}
in which case the redundant part is given by $\ker A \cap \ker B=\ker Q$.
Then the relation $L(A,B)$ is closed and  the orthogonal  projection
from $\sH \times \sK$ onto $L(A,B)$  is given by
\begin{equation}\label{rutte0}
 P_{L(A,B)}
  =\begin{pmatrix}
                 AA^* & AB^* \\
                 BA^* & BB^*
               \end{pmatrix}.
 \end{equation}
Moreover, the orthogonal  projection
from $\sH \times \sK$ onto $L(A,B)^*$  is given by
\begin{equation}\label{rutte1}
 P_{L(A,B)^*}=\begin{pmatrix} I-BB^* & BA^*\\ A B^* & I-AA^*
 \end{pmatrix}.
\end{equation}
\end{lemma}

\begin{proof}
Since $\ran Q$ is closed, it follows from Lemma \ref{Tclosed} that $L(A,B)$ is closed.
Let $\Phi=c(A,B)$, so that $\Phi \in \mathbf{B}(\sE, \sH \times \sK)$. Then \eqref{isomTabQ}
means that $\Phi^*\Phi=Q$ and $\ran \Phi^*=\ran Q$. Observe that
\[
 (\Phi \Phi^*) (\Phi \Phi^*)= \Phi (\Phi^* \Phi) \Phi^*=\Phi Q \Phi^*=\Phi \Phi^*,
\]
so that $\Phi \Phi^*$ is an orthogonal projection, mapping onto $\ran \Phi \Phi^*=\ran \Phi$.
Thus the statement follows from \eqref{column2};   cf. \cite[Appendix D, p. 703]{BHS}.

 Note that  the condition \eqref{isomTabQ} implies that
$\ker A \cap \ker B=\ker \Phi=\ker Q$.  Hence it follows that $\ran r(A^{*}, B^{*})$
 is dense in $\sE_0$ and thus $\ran r(A^{*}, B^{*})=\sE_0=\ran Q$
 by Lemma \ref{Tclosed0}.
Note that therefore  also $\ran r (B^{*}, -A^{*})=\sE_0$,
 so that the selfadjoint mapping
\[
 \begin{pmatrix} B \\ -A \end{pmatrix}
 \begin{pmatrix} B \\ -A \end{pmatrix}^*
 =\begin{pmatrix} B \\ -A \end{pmatrix} (B^*\,\, -A^*)
\]
takes $\sK \times \sH$ onto $JL(A,B)$. Moreover, this mapping is,
due to \eqref{isomTabQ}, idempotent. Thus the orthogonal projection
from $\sK \times \sH$ onto $JL(A,B)$ is given by
\[
 P_{JL(A,B)}=\begin{pmatrix} B \\ -A \end{pmatrix}
 \begin{pmatrix} B \\ -A \end{pmatrix}^*
 =\begin{pmatrix} BB^* & -BA^*\\ -AB^* & AA^* \end{pmatrix},
\]
which is equivalento to \eqref{rutte0}.

Since $L(A,B)^*=(JL(A,B))^\bot$,  the orthogonal projection
onto $L(A,B)^* $ is given by
\[
 P_{L(A,B)^*}=I-P_{JL(A,B)}
 =\begin{pmatrix} I-BB^* & BA^*\\ A B^* & I-AA^*
 \end{pmatrix},
\]
which gives \eqref{rutte1}.
 \end{proof}

Recall that the adjoint relation $L(A,B)^*$ has the representation \eqref{Tab*}.
As a closed linear relation from $\sK$ to $\sH$ it is an operator range relation.
Under the condition \eqref{isomTabQ}  such a parametrization can be made explicit via
Lemma \ref{Tclosed0Q}.

\begin{corollary}\label{*Repres}
Let the linear relation $L(A,B)$ be given by \eqref{Tabegin}
and assume that  the condition \eqref{isomTabQ} holds.
Then  the adjoint relation $L(A,B)^*$ is given by
 \[
  L(A,B)^*
=\bigl\{ \{  (I-BB^*)\varphi +BA^* \psi,
\, A B^* \varphi +(I-AA^*) \psi \}:\,
\varphi \in \sK, \, \psi \in \sH \bigr\}.
 \]
 \end{corollary}

  The pair of bounded linear operators $A\in\bB(\sE,\sH)$ and
$B\in\bB(\sE,\sK)$
  in \eqref{Tabegin} is said to be \textit{normalized} if
\begin{equation}\label{isomTab}
 A^*A+B^*B=I.
\end{equation}
 In this case, the relation $L(A,B)$ is closed by Lemma \ref{Tclosed0Q}
and the representation of $L(A,B)$ in \eqref{Tabegin}
is reduced.

\begin{lemma}\label{Lclosed}
Let the linear relation $L(A,B)$ be given by \eqref{Tabegin}
and assume that $L(A,B)$ is closed. Then $L(A,B)$
can be represented by a normalized pair.
 \end{lemma}

\begin{proof}
 Since $L(A,B)$ is closed,
 it is known by Lemma \ref{Tclosed0}
 that $\ran (A^*A+B^*B)$ is closed.
 Thus for the reduced representation in \eqref{Tabegin-} one  now has
\[
\ran ((A_0)^*A_0+(B_0)^*B_0)=\sE_0
\]
by Lemma \ref{Tclosed0n}.
Replacing $A_0$ and $B_0$ by the equivalent pair
\[
 A_0'=A_0((A_0)^*A_0+(B_0)^*B_0)^{-\half}
 \quad \mbox{and} \quad B_0'=B_0((A_0)^*A_0+(B_0)^*B_0)^{-\half}
\]
leads to   $L(A,B)=L(A_0',B_0')$, where $A_0'\in\bB(\sE_0,\sH)$ and
$B_0'\in\bB(\sE_0,\sK)$ form a normalized pair.
\end{proof}

 If $A\in\bB(\sE,\sH)$ and $B\in\bB(\sE,\sK)$ are  normalized as in \eqref{isomTab},
then the orthogonal operator part of $L(A,B)=BA^{-1}$
can be rewritten  in terms of the \textit{Moore-Penrose inverse} $A^{(-1)}$
of $A$, which is an operator from $\ran A \subset \sH$ to $(\ker A)^\perp \subset \sE$.
In fact, it is defined as follows: for $\psi \in \ran A$, there exists a unique $\varphi \in (\dom A)^\perp$
such that $A\varphi=\psi$, and $A^{(-1)} \psi := \varphi$; cf. \cite{BHS}.
Note that in the literature one often also extends $A^{(-1)}$ by $A^{(-1)} \psi=0$ for $\psi \in (\ran A)^\perp$.

\begin{lemma}\label{*Repress}
Let the linear relation $L(A,B)$ be given by \eqref{Tabegin}
and assume that  the condition \eqref{isomTab} holds.
Then the orthogonal operator part $L(A,B)_{\rm s}$ is given by
\begin{equation}\label{ss}
 L(A,B)_{\rm s}=\{\,\{A  \varphi, B \varphi \}:\,
 \varphi \in (\ker A)^{\perp}  \}, 
\end{equation}
and, consequently,
\begin{equation}\label{ssa}
 L(A,B)_{\rm s} 
 =B A^{(-1)}.
\end{equation}
\end{lemma}

\begin{proof}
Consider the representation \eqref{Tabegin}. Since $A \in \bB( \sE, \sH)$,
the orthogonal decomposition
\[
 \sE=(\ker A)^{\perp} \oplus \ker A
\]
leads to the alternative representation
\begin{equation}\label{aabb1}
 L(A,B)
 =\{\,\{A  \varphi, B \varphi+B \psi\}:\,
 \varphi \in (\ker A)^{\perp},\, \psi\in \ker A\,\},
\end{equation}
 cf. \cite{HSS2018}.
Note that for all 
$\varphi \in (\ker A)^{\perp}$ and $\psi \in \ker A$, one has by \eqref{conn}
\[
\begin{split}
 (B \varphi, B \psi)&=( B^*B \varphi, \psi)
 =( (I-A^*A) \varphi, \psi) \\
& =(\varphi,\psi)- (A \varphi, A \psi)=0.
\end{split}
\]
Consequently, the range decomposition in the representation \eqref{aabb1} is orthogonal.
Recall from \eqref{Tabb} that $\mul L(A,B)=B(\ker A)$.
Thus the orthogonal operator part of $L(A,B)$ has the representation \eqref{ss}.
The representation \eqref{ssa} follows from the definition of $A^{(-1)}$.
\end{proof}

 If $A\in\bB(\sE,\sH)$ and $B\in\bB(\sE,\sK)$ satisfy \eqref{isomTabQ},  
then the statement in
Corollary \ref{*Repres} is a direct consequence of Lemma \ref{Tclosed0Q}.
Under the stronger condition \eqref{isomTab},
a slightly different looking result can be obtained 
by invoking the polar decomposition of the operator $A$, i.e.,
\[
A=V_A(A^*A)^{1/2},
\]
where $V_A$ is the unique partial isometry from $\sE$ to $\sH$ with initial space $\cran A^*$ and final space $\cran A$.
Recall that $I-V_A (V_A)^*=P_{\,\ker A^*}$.
By means of the polar decomposition of $A$ one sees that
the normalization  \eqref{isomTab} leads to
\[
\begin{split}
 I-AA^*&=I-V_A A^*A(V_A)^*\\
 &=I-V_A (V_A)^*+V_AB^*B(V_A)^*=P_{\,\ker A^*}+V_AB^*B(V_A)^*.
\end{split}
\]
Moreover, by the well-known commutation relations
\[
B(I-B^*B)^{1/2}=(I-BB^*)^{1/2}B, \quad
(I-B^*B)^{1/2}B^*=B^*(I-BB^*)^{1/2},
\]
one obtains from the normalization  \eqref{isomTab} that
\[
AB^{*}=V_A(A^*A)^{1/2}B^{*}=V_{A} (I-B^{*}B)^{{1/2}}B^{*}=V_{A} B^{*}(I-BB^{*})^{{1/2}}.
\]
Therefore the orthogonal  projection
from $\sH \times \sK$ onto $L(A,B)^*$ in Lemma \ref{Tclosed0Q}   is given by
 \[
\begin{split}
 P_{L(A,B)^*}& =\begin{pmatrix} I-BB^* & (I-BB^*)^{1/2}B(V_A)^*\\
 V_AB^*(I-BB^*)^{1/2} & P_{\,\ker A^*}+ V_AB^*B(V_A)^* \end{pmatrix}\\
 & =\begin{pmatrix} (I-BB^*)^{1/2}\\ V_AB^* \end{pmatrix}
 \begin{pmatrix} (I-BB^*)^{1/2}\\ V_AB^* \end{pmatrix}^*
    + \begin{pmatrix} 0 & 0\\ 0 & P_{\,\ker A^*}\end{pmatrix}.
\end{split}
\]
From this, the following result is now clear.

\begin{corollary}\label{starstar}
Let the linear relation $L(A,B)$ be given by \eqref{Tabegin}
and assume that  the condition \eqref{isomTab} holds.  Then
\[
 L(A,B)^*=\{\,\{(I-BB^*)^{1/2}k,V_A B^*k+ P_{\,\ker A^*}h\}:\,
 h\in\sH,\, k\in \sK\,\},
\]
where the range decomposition is orthogonal.
\end{corollary}

Note that the results in Corollary \ref{*Repres} and Corollary \ref{starstar}   are of practical importance
in the description of boundary value problems (when one speaks of boundary values in parametrized from);
cf. \cite{BHS}.

\section{Construction of the closure of operator range relations}\label{sec5}

Let $A\in\bB(\sE,\sH)$ and $B\in\bB(\sE,\sK)$.
Then the linear relation $L(A,B)$ from $\sH$ to $\sK$  in \eqref{Tabegin}
is an operator range relation.
By Lemma \ref{cloclo} the closure $L(A,B)^{**}$
of $L(A,B)$ is an operator range relation  and
 it will be shown how the closure $L(A,B)^{**}$ can be represented
in terms of the original operators $A$ and $B$.

\vspace{0.2cm}

First return to the treatment involving the pair of bounded operators $A$ and $B$.
 The following construction is inspired by  arguments appearing in measure theory.
Due to the obvious inequalities
 \begin{equation}\label{eqCACB-}
A^*A\le A^*A+B^*B, \quad B^*B\le A^*A+B^*B,
\end{equation}
an application of the Douglas lemma \cite{Doug} shows that there exists a pair of contractions
$C_A \in \bB( \sE, \sH)$ and $C_B \in \bB( \sE, \sK)$,
such that
\begin{equation}\label{eqCACB++}
 A=C_A\, (A^*A+B^*B)^{1/2}, \quad B=C_B\,(A^*A+B^*B)^{1/2},
\end{equation}
or, equivalently
\begin{equation}\label{eqCACB}
 A^*=(A^*A+B^*B)^{1/2} (C_A)^*, \quad B^*=(A^*A+B^*B)^{1/2} (C_B)^*.
\end{equation}
The contractions
$C_A  \in \bB(\sE, \sH)$ and $C_B \in \bB( \sE , \sK)$
are uniquely determined by the conditions that $C_A$ and $C_B$
vanish on $(\ran (A^*A+B^*B)^\half)^\perp=\ker A \cap \ker B$
or, equivalently,
\begin{equation}\label{eqCACBa}
 \ran (C_A)^*\subset \cran (A^*A+B^*B), \quad
 \ran (C_B)^*\subset \cran (A^*A+B^*B),
\end{equation}
and with these conditions one also has as a consequence
\[
 \ker (C_A)^*=\ker A^*, \quad \ker (C_B)^*=\ker B^*.
\]
It follows from \eqref{eqCACB++} that $L(A,B)$
can be written as
\begin{equation}\label{eqCACB1}
 L(A,B)=\{\,\{C_A \, (A^*A+B^*B)^{1/2} f,C_B \,
 (A^*A+B^*B)^{1/2}f\}:\, f\in\sE\,\}.
 \end{equation}
This representation will be called the
\textit{canonical representation} of $L(A,B)$.
Of course the pair $C_A$ and $C_B$ itself induces a linear
relation $L(C_A,C_B)$ in $\sH \times \sK$, in other words,
\begin{equation}\label{eqCACB2-}
L(C_A,C_B)=\{\,\{C_A f,C_B \, f\}:\, f\in \sE\,\}.
\end{equation}

Analogous to  \eqref{column} one may now introduce the column operator
$c(C_A,C_B)$ from $\sE$ to $\sH \times \sK$ by
\begin{equation}\label{column11}
c(C_A, C_B) = \begin{pmatrix} C_A  \\ C_B \end{pmatrix},
\quad \mbox{i.e.,} \quad
 c(C_A, C_B) \varphi= \begin{pmatrix} C_A \varphi \\ C_B \varphi \end{pmatrix},
\quad \varphi \in \sE.
\end{equation}
Then $c(C_A,C_B)$ belongs to $\bB(\sE,\sH\times\sK)$ and
\[
L(C_A,C_B)=\ran c(C_A,C_B);
\]
cf. \eqref{Tab**}.  It follows from \eqref{eqCACB1} that
\begin{equation}\label{rprime}
 L(A,B)=\ran c(A,B) \subset \ran c(C_A,C_B)
 =L(C_A,C_B).
\end{equation}
Note that if $L(A,B)$ is closed, then $L(C_A,C_B)=L(A,B)$
by \eqref{eqCACB1}, \eqref{column1+}, and Lemma \ref{Tclosed0}.

\medskip

The next lemma contains a first step in establishing the connection between
the pairs $A,B$ and $C_A, C_B$.

\begin{lemma}\label{newgri}
Let $A\in\bB(\sE,\sH)$, $B\in\bB(\sE,\sK)$, and
let $C_A$, $C_B$ be uniquely defined by
\eqref{eqCACB++} and \eqref{eqCACBa}.
Then
\begin{equation}\label{conn000}
  \ker C_A \cap \ker C_B =\ker A \cap \ker B,
\end{equation}
and
\begin{equation}\label{conn}
 [(C_A)^*C_A+(C_B)^*C_B]  h
 = \left\{ \begin{array}{ll} h, &  h \in  \cran (A^*A+B^*B),\\
                                       0, &   h \in \ker A \cap \ker B.
             \end{array}
             \right.
\end{equation}
\end{lemma}

\begin{proof}
It follows from  \eqref{eqCACB++}, \eqref{eqCACB}, and \eqref{eqCACBa}  that
\[
(A^*A+B^*B)^{1/2} [(C_A)^* C_A+(C_B)^* C_B](A^*A+B^*B)^{1/2}
=A^*A+B^*B.
\]
 Consequently, thanks to \eqref{eqCACBa},
\[
[(C_A)^* C_A+(C_B)^* C_B](A^*A+B^*B)^{1/2} =(A^*A+B^*B)^{1/2},
\]
which by continuity of $(C_A)^* C_A+(C_B)^* C_B$ implies that
\begin{equation}\label{conn-}
 [(C_A)^*C_A+(C_B)^*C_B]  h =h, \quad h \in  \cran (A^*A+B^*B).
\end{equation}
Hence the first part of \eqref{conn} holds.

Keeping  \eqref{Tabbb} in mind, it is clear that in the present context one also has
\[
 \sE=\cran( (C_A)^* C_A+(C_B)^* C_B ) \oplus (\ker C_A \cap \ker C_B).
\]
Observe that it follows from  \eqref{eqCACBa}  that
\[
\cran( (C_A)^* C_A+(C_B)^* C_B ) \subset \cran (A^*A+B^*B).
\]
or, equivalently,
\[
\ker A \cap \ker B =\ker (A^*A+B^*B) \subset \ker C_A \cap \ker C_B.
\]
Assume that these inclusions are strict. Then there exists a nontrivial element
$h \in \cran (A^*A+B^*B)$
with $h \in \ker C_A \cap \ker C_B$, which contradicts \eqref{conn-}. Thus
\[
\cran( (C_A)^* C_A+(C_B)^* C_B ) = \cran (A^*A+B^*B),
\]
and \eqref{conn000}  follows. Thus also the second part of \eqref{conn} holds.
\end{proof}

Note that the pair $C_A$ and $C_B$ is normalized precisely when it is reduced
or, equivalently, when the pair $A$ and $B$ is reduced.

\medskip

The precise connection between the pairs $A, B$ and $C_A, C_B$ is now established
in terms of  the \emph{polar decomposition} (cf. \cite[Section~VI~2.7]{Kato})
of the column operator $c(A,B)$   in \eqref{column}; cf. \eqref{eqCACB}.

\begin{lemma}\label{polar}
Let $A\in\bB(\sE,\sH)$, $B\in\bB(\sE,\sK)$, and
let $C_A$, $C_B$ be uniquely defined by
\eqref{eqCACB++} and \eqref{eqCACBa}.
Let $c(A,B)$ and $c(C_A,C_B)$ be as defined in  \eqref{column}
and \eqref{column11}, respectively.
Then the identity
 \[
 \begin{pmatrix} A \\ B \end{pmatrix} = \begin{pmatrix} C_A \\ C_B \end{pmatrix}  (A^*A+B^*B)^{1/2},
\]
is the polar decomposition of the column operator $c(A,B)$,
where the column operator $c(C_A,C_B)$ is the  unique  partial isometry with
initial space $\cran (A^*A+B^*B)$ and final space $\cran c(A,B)$.
\end{lemma}

\begin{proof}
By Lemma \ref{newgri} $(C_A)^* C_A+(C_B)^* C_B$ is the orthogonal projection from $\sE$ onto
\[
 \ran ((C_A)^* C_A+(C_B)^* C_B)=\cran (A^*A+B^*B).
\]
In particular, $\ran  [(C_A)^*C_A+(C_B)^*C_B]$ is closed and, equivalently,
$\ran c(C_A, C_B)=L(C_A, C_B)$ is closed; cf. Lemma \ref{Tclosed0}.
Then the column operator $c(C_A, C_B)$ is a partial isometry with initial space
$\cran (A^*A+B^*B)$ and final space $\ran c(C_A, C_B)$ (cf. \cite[Appendix~D]{BHS}).
Since $\ran (A^*A+B^*B)^{1/2}$ is dense in $\cran (A^*A+B^*B)$, also its image under
$c(C_A, C_B)$ is dense in $\ran c(C_A, C_B)$. By construction
$c(C_A, C_B)$ maps $\ran (A^*A+B^*B)^{1/2}$ onto $\ran c(A,B)$, cf. \eqref{eqCACB1},
and thus the final space of $c(C_A, C_B)$  is equal to $\cran c(A, B)$.
\end{proof}

Recall that the reduction of $L(A,B)$ in \eqref{uno}
is with respect to the orthogonal decomposition \eqref{Tabbb}
of the space $\sE$:
\begin{equation*}\label{ggrrii}
\sE=\sE_0 \oplus (\ker A \cap \ker B) \quad \mbox{with} \quad \sE_0=\cran (A^*A+B^*B)^{1/2}.
\end{equation*}
Thus  $A$ and $B$ are row operators:
\begin{equation}\label{reduu0}
A=r(A_0, O_{\ker A \, \cap \, \ker B}) \quad \mbox{and} \quad B=r(B_0, O_{\ker A \, \cap \, \ker B}),
\end{equation}
such that $A_0 \in \mathbf{B}(\sE_0, \sH)$ and $B_0 \in \mathbf{B}(\sE_0, \sK)$.
 As in  \eqref{eqCACB++}, there exists a pair of   operators
$C_{A_0} \in \mathbf{B}(\sE_0, \sH)$ and $C_{B_0}  \in \mathbf{B}(\sE_0, \sK)$, such that
\begin{equation}\label{reduu}
 A_0= C_{A_0}  ((A_0)^*A_0+(B_0)^*B_0)^\half, \quad B_0= C_{B_0}  ((A_0)^*A_0+(B_0)^*B_0)^\half,
\end{equation}
and these operators are unique.
 It will be shown that, in fact,  the operators $C_{A_0}$ and  $C_{B_0}$ are the reductions of the pair
$C_A$ and $C_B$ in \eqref{eqCACB++}.

\begin{lemma}\label{rreduxx}
Let $A\in\bB(\sE,\sH)$, $B\in\bB(\sE,\sK)$,
let $C_A$, $C_B$ be uniquely defined by
\eqref{eqCACB++} and \eqref{eqCACBa},
 and let $A_0$ and $B_0$ be the restrictions of $A$ and $B$ as in \eqref{uno}.
Then the operators $C_{A_0}$ and $C_{B_0}$ in \eqref{reduu} are the restrictions of $C_A$ and $C_B$.
In other words:
 \[
L(C_A,C_B) =\{\,\{C_{A_0} f, C_{B_0} \, f\}:\, f\in \sE_0\,\}
 \]
is a  reduced representation of $L(C_A,C_B)$.
\end{lemma}

\begin{proof}
 With the restrictions $A_0$ and $B_0$, it follows from \eqref{duo} and \eqref{eqCACB++}   that
\begin{equation}\label{duo+}
\begin{split}
 A&=C_A (A^*A+B^*B)^\half=C_A \begin{pmatrix} ((A_0)^*A_0+(B_0)^*B_0)^\half & 0 \\0 &0\end{pmatrix}, \\
 B&=C_B (A^*A+B^*B)^\half=C_B \begin{pmatrix} ((A_0)^*A_0+(B_0)^*B_0)^\half & 0 \\0 &0\end{pmatrix}.
\end{split}
\end{equation}
It has been shown in \eqref{conn000} that $\sE_0=\ker A \cap \ker B=\ker C_A \cap C_B$.
Now let $\wt C_A \in \mathbf{B}(\sE_0, \sH)$ and $\wt C_B \in \mathbf{B}(\sE_0, \sK)$
be the restrictions of $C_A$ and $C_B$,
so that $C_A=(\wt C_A ,0)$ and $C_B=(\wt C_B ,0)$.
Then with $A=r(A_0 , 0)$ and $B=r(B_0 , 0)$,  it follows
from \eqref{duo+}, that
\[
 A_0=\wt C_A ((A_0)^*A_0+(B_0)^*B_0)^\half, \quad B_0=\wt C_B ((A_0)^*A_0+(B_0)^*B_0)^\half.
\]
A comparison with \eqref{reduu} gives that $\wt C_A=C_{A_0}$ and $\wt C_B=C_{B_0}$.
\end{proof}

\medskip

The description of $L(A,B)^{**}$ and its orthogonal operator part
follows easily from Lemma \ref{polar} and Lemma \ref{*Repress}.
The Moore-Penrose inverse $(C_A)^{(-1)}$ of $C_A$
is an operator from $\ran C_A \subset \sH$ to $(\ker C_A)^\perp \subset \sE$,
which is defined as follows: for $\psi \in \ran C_A$,
there exists a unique $\varphi \in (\dom C_A)^\perp$
such that $C_A \varphi=\psi$, and $(C_A)^{(-1)} \psi := \varphi$; cf. \cite{BHS}.
Note that in the literature one often also extends $A^{(-1)}$
by $A^{(-1)} \psi=0$ for $\psi \in (\ran A)^\perp$.

\begin{theorem}\label{Main}
Let the linear relation $L(A,B)$ be given by \eqref{Tabegin}
and let $C_{A}$ and $C_{B}$ be the uniquely defined operators
satisfying \eqref{eqCACB++} and \eqref{eqCACBa}.
Then the linear relation $L(C_A,C_B)$ is closed and
\begin{equation}\label{Mainx}
L(A,B)^{**}=L(C_A,C_B)=\{\,\{C_A  \varphi,C_B \varphi \}:\,
 \varphi \in \sE \} =C_B(C_A)^{-1}.
\end{equation}
Consequently, the orthogonal operator part of $L(A,B)^{**}$ is given by
\begin{equation}\label{Mainy}
 (L(A,B)^{**})_{\rm s}=\{\,\{C_A  \varphi,C_B \varphi \}:\,
 \varphi \in (\ker C_{A})^{\perp} \}=C_B(C_A)^{(-1)}.
\end{equation}
 \end{theorem}

\begin{proof}
 As shown in Lemma \ref{polar} one has $\ran c(C_A,C_B)=\cran c(A,B)$.
Hence, it follows from \eqref{rprime}, after taking closures, that
\[
 L(A,B)^{**}=\cran c(A,B) = \ran c(C_A,C_B)=L(C_A,C_B),
\]
and it is clear that $L(C_A,C_B)=C_B(C_A)^{-1}$. Thus \eqref{Mainx} has been shown.

 In order to show \eqref{Mainy}, observe that $C_A \in \bB( \sE, \sH)$ gives
the orthogonal decomposition
\[
  \sE=(\ker C_{A})^{\perp} \oplus \ker C_{A}.
\]
Hence, \eqref{Mainx} leads to the alternative representation
\begin{equation}\label{aabb1+}
 L(C_A,C_B)
 =\{\,\{C_A  \varphi,C_B \varphi+C_{B}\psi\}:\,
 \varphi \in (\ker C_{A})^{\perp},\, \psi\in \ker C_A\,\}.
\end{equation}
It follows from \eqref{eqCACBa} that
\[
 (\ker C_A)^\perp=\cran (C_A)^*\subset \cran (A^*A+B^*B).
\]
Therefore,  for all
$\varphi \in (\ker C_{A})^{\perp}$ and $\psi \in \ker C_A$ one has by \eqref{conn}
\[
\begin{split}
 (C_B \varphi, C_B \psi)&=( (C_B)^*C_B \varphi, \psi)
 =( (I-(C_A)^*C_A) \varphi, \psi) \\
& =(\varphi,\psi)- (C_A \varphi, C_A \psi)=0.
\end{split}
\]
Consequently, the range decomposition in the representation \eqref{aabb1+} is orthogonal.
Furthermore, recall that
\[
 \mul L(C_A,C_B)=C_B( \ker C_A),
\]
cf. \eqref{Tabb}, where
the right-hand side is closed as the left-hand side is closed.
Thus the orthogonal operator part of $L(C_A,C_B)$ has the representation
\[
 L(C_A,C_B)_{\rm s}=\{\,\{C_A  \varphi,C_B \varphi \}:\,
 \varphi \in (\ker C_{A})^{\perp} \},
\]
which proves \eqref{Mainy}.
\end{proof}

Finally, the representation of the adjoint relation $L(C_A,C_B)^*$
will be considered.
Recall that the identity \eqref{Tab*} gives the following representation:
\[
  L(C_A,C_B)^*=\{\,\{k,h\}\in \sK\times\sH:\, (C_B)^*k=(C_A)^*h\,\}.
\]
An application of  Lemma \ref{Tclosed0Q} and Lemma \ref{newgri}
gives the following result; cf. Corollary \ref{*Repres}.

\begin{corollary}\label{*Represcc}
Let the linear relation $L(A,B)$ be given by \eqref{Tabegin}
and let $C_{A}$ and $C_{B}$ be the uniquely defined operators
satisfying \eqref{eqCACB++} and \eqref{eqCACBa}.
 Then  the adjoint relation $L(C_A,C_B)^*$ is given by
 \[
 \bigl\{ \{  (I-C_B (C_B)^*)\varphi +C_B (C_A)^* \psi,
\, C_A (C_B)^* \varphi +(I-C_A (C_A)^*) \psi \}:\,
\varphi \in \sK, \, \psi \in \sH \bigr\}.
 \]
 \end{corollary}

 In order to rewrite the result  in Corollary \ref{*Represcc},
consider the polar decomposition
of the contraction $C_A  \in \bB(\sE, \sH)$:
\[
C_A=W_A\big((C_A)^*C_A\big)^{1/2},
\]
where $W_A$ is the unique partial isometry from $\sH_{A,B}$ to $\sH$
with initial space and final space given by
\[
 (\ker W_A)^\perp
 =\cran \big((C_{A})^{*}C_{A}\big)^{1/2}
 \quad \mbox{and} \quad \ran W_A=\cran C_A,
\]
respectively. Under the assumption that the pair $A$ and $B$ is reduced
it follows that the pair $C_A$ and $C_B$ is normalized, see Lemma \ref{newgri}.
Hence the argument preceding Corollary \ref{starstar} now gives the following result.

\begin{corollary}
Let the linear relation $L(A,B)$ be given by \eqref{Tabegin}
and let $C_{A}$ and $C_{B}$ be the uniquely defined operators
satisfying \eqref{eqCACB++} and \eqref{eqCACBa}.
Assume that the pair $A$ and $B$ is reduced.
Then the adjoint relation $L(C_A,C_B)^*$ is given by
\[
 L(C_A,C_B)^*
 =\{\{\big(I-C_B(C_B)^*\big)^{1/2}h,W_A (C_B)^*h+ k \}:\, h\in \sK, \, k \in\ker (C_A)^* \},
\]
where the range decomposition is orthogonal.
\end{corollary}

\section{Regular and singular operator range relations}\label{sec6}

Let $\sE$, $\sH$, and $\sK$ be Hilbert spaces
and consider the pair of bounded linear operators $A\in\bB(\sE,\sH)$ and
$B\in\bB(\sE,\sK)$.
Let the operator range relation $L(A,B)$
be given by \eqref{Tabegin}.
Criteria will be given for the regularity and singularity of $L(A,B)$
in terms of $A~$ and $B$, and in terms of $C_A$ and $C_B$ in \eqref{eqCACB++}.

\medskip

First the regularity and the singularity of
$L(A,B)$ will be expressed in terms of the pair $A$ and $B$; see \cite{HSS2018}
for some further equivalent statements which hold for general linear relations.
The regularity of $L(A,B)$ is described by the following lemma.

\begin{lemma}\label{TAB}
Let the linear relation $L(A,B)$ and $\sD(A,B)$ be given
by \eqref{Tabegin} and \eqref{RAB}, respectively.
Then the following statements are equivalent:
\begin{enumerate}\def\labelenumi{\rm(\roman{enumi})}
\item $L(A,B)$ is regular;
 \item the set $\sD(A,B)$
 is dense in $\sK$.
\end{enumerate}
Moreover, the following statements are equivalent:
\begin{enumerate}\def\labelenumi{\rm(\roman{enumi})}\setcounter{enumi}{2}
\item $L(A,B)$ is a bounded operator;
 \item  $\sD(A,B)=\sK$, i.e., $\ran B^*\subset \ran A^*$.
\end{enumerate}
Finally, the following statements are equivalent:
\begin{enumerate}\def\labelenumi{\rm(\roman{enumi})}\setcounter{enumi}{4}
\item $L(A,B) \in \bB(\sH,\sK)$.
 \item  $\ran B^*\subset \ran A^*$ and $\ran A =\sH$.
\end{enumerate}
\end{lemma}

\begin{proof}
(i) $\Leftrightarrow$ (ii)
Recall that $L(A,B)$ is regular if and only if
$\mul L(A,B)^{**}=\{0\}$.  It follows from \eqref{tabst1}
that this is equivalent to $\sD(A,B)$ being dense in $\sK$; see \eqref{tabst}.
The other description for $\sD(A,B)$ is obtained from \eqref{RAB}.

(iii) $\Leftrightarrow$ (iv)
The relation $L(A,B)$ is a bounded operator precisely
if there exists $c \ge 0$ such that
\begin{equation}\label{Tabopp}
 \|Bf\| \le c \|Af\|, \quad f \in \sE,
\end{equation}
or, equivalently, if $B^{*}B \le c^2 A^{*}A$.
By the Douglas lemma \cite{Doug} this is equivalent to $\ran B^*\subset \ran A^*$.
This last inclusion is the same as saying $\dom L(A,B)^*=\sD(A,B)=\sK$;
see \eqref{RAB} and \eqref{tabst}.

(v) $\Leftrightarrow$ (vi) This is now clear.
\end{proof}

Likewise, the singularity of $L(A,B)$ can be expressed
in various equivalent useful ways.

\begin{lemma}\label{labsing}
Let the linear relation $L(A,B)$ be given by \eqref{Tabegin}.
Then the following statements are equivalent:
\begin{enumerate}\def\labelenumi{\rm(\roman{enumi})}
\item $L(A,B)$ is singular;
\item $L(A,B)^{**}=\cran A \times \cran B$;
 \item $\sD(A,B) \subset \ker B^*$;
 \item $\sR(A,B) \subset \ker A^*$;
\item $\ran A^*\cap\ran B^*=\{0\}$;
\item $L(A,B)^*=\ker B^* \times \ker A^*$.
\end{enumerate}
\end{lemma}

\begin{proof}
This is a straightforward application of \cite[Proposition 2.8]{HSS2018},
together with \eqref{tabst} and \eqref{tabst+}.
 It suffices to prove (iii) $\Leftrightarrow$ (v).

(iii) $\Rightarrow$ (v) Let  $\ell \in \ran A^*\cap\ran B^*$.
Then $\ell \in \ran A^*$ and $\ell=B^{*}h$ for some $h \in \sH$.
Thus $h \in \sD(A,B)$ and by assumption one sees that $h \in \ker B^*$.
Thus it follows that $\ell =0$, and hence (v) holds.

(v) $\Rightarrow$ (iii) This implication is trivial.
\end{proof}

Next the regularity and the singularity of
$L(A,B)$ will be expressed in terms of the pair
 $C_A$ and $C_B$ in \eqref{eqCACB++}.
Recall from Theorem \ref{Main} that $L(A,B)^{**}=L(C_A,C_B)$.
Thus the following characterization of the regularity of $L(A,B)$ is clear.

\begin{lemma}\label{maincor}
Let the linear relation $L(A,B)$ be given by \eqref{Tabegin}
and let $C_{A}$ and $C_{B}$ be the uniquely defined operators
satisfying \eqref{eqCACB++} and \eqref{eqCACBa}.
Then the following statements are equivalent:
\begin{enumerate}\def\labelenumi{\rm(\roman{enumi})}
\item $L(A,B)$ is regular;
\item $L(C_{A},C_{B})$ is an operator;
\item $\ker C_A \subset \ker C_B$.
\end{enumerate}
\end{lemma}

Likewise, the singularity of $L(A,B)$ can be expressed
as follows.

\begin{corollary}
Let the linear relation $L(A,B)$ be given by \eqref{Tabegin}
and let $C_{A}$ and $C_{B}$ be the uniquely defined operators
satisfying \eqref{eqCACB++} and \eqref{eqCACBa}.
Then the following statements are equivalent:
\begin{enumerate}\def\labelenumi{\rm(\roman{enumi})}
\item $L(A,B)$ is singular;
\item  $\ker C_A + \ker C_B=\sE$.
\end{enumerate}
\end{corollary}

\begin{proof}
 Due to $L(A,B)^{**}=L(C_A,C_B)$ (see Theorem \ref{Main}),
it follows from  Lemma \ref{labsing} that $L(A,B)$ is singular
if and only if $L(C_A,C_B)=\ran C_A\times \ran C_B$,
or, equivalently,
precisely when
\begin{equation}\label{ranSing}
  \ran c(C_A,C_B)
 =\ran C_A\times \ran C_B.
\end{equation}

(i) $\Rightarrow$ (ii) Assume that $L(A,B)$ is singular. Then, thanks to \eqref{ranSing},
 for every $e \in \sE$ there exist $e_1, e_2 \in \sE$ such that
\[
 c(C_A,C_B) e=\begin{pmatrix} h_1 \\ h_2 \end{pmatrix},
 \quad
 c(C_A,C_B) e_1=\begin{pmatrix} h_1 \\ 0 \end{pmatrix},
\quad
c(C_A,C_B) e_2=\begin{pmatrix} 0 \\ h_2 \end{pmatrix}.
\]
Then, clearly, $e_1\in\ker C_B$, $e_2\in\ker C_A$, and
\[
e_0=e-e_1-e_2\in \ker c(C_A,C_B)=\ker C_A\cap \ker C_B.
\]
This shows that $e \in \ker C_A+\ker C_B$. Therefore, $\sE=\ker C_A+\ker C_B$.

(ii) $\Rightarrow$ (i)
Assume that  $\sE=\ker C_A+\ker C_B$.  Let $a=a_1+a_2$ and $b=b_1+b_2$ with
$a_1,b_1 \in \ker C_B$ and $a_2,b_2 \in \ker C_A$. Then the equality
\[
 \begin{pmatrix} C_A a \\ C_B b \end{pmatrix} = \begin{pmatrix} C_A a_1 \\ C_B b_2 \end{pmatrix}
 = \begin{pmatrix} C_A (a_1+b_2) \\ C_B (a_1+b_2) \end{pmatrix},
\]
shows that $\ran C_A\times \ran C_B  \subset \ran c(C_A,C_B)$.
Hence, \eqref{ranSing} is satisfied.
\end{proof}

\section{Classification of pairs of bounded linear operators}\label{sec7}

Let $\sE$, $\sH$, and $\sK$ be Hilbert spaces and let $A\in\bB(\sE,\sH)$
and $B\in\bB(\sE,\sK)$ be a pair of bounded linear operators.
The characterizations of the operator range relation $L(A,B)$  from \eqref{Tabegin}
will now be augmented by further characterizations in terms of $A$ and $B$, that
are influenced by similar observations in measure theory.

\begin{definition}\label{D}
Let $\sE$, $\sH$, and $\sK$ be Hilbert spaces and let $A\in\bB(\sE,\sH)$
and $B\in\bB(\sE,\sK)$. Then the operator $B$ is said to be \textit{dominated}
by $A$, denoted by $B \prec A$, if there exists some $c>0$ such that
\[
 \|B f\| \leq c \|Af\| \quad \mbox{for all} \quad f \in \sE.
\]
\end{definition}

By the Douglas lemma this definition is equivalent to the factorization
$B=CA$ where $C \in \bB(\sH, \sK)$;
 the operator $C$ is uniquely determined when $\ran C^* \subset \cran A$, in
  which case $\ker C^* =\ker B^*$.
Note that in the present  context Definition \ref{D} agrees
with the definition of domination in \cite{HSn2015}.
For the following, it is useful to recall from the Douglas lemma that $B$ is dominated
by $A$ if and only if $\ran B^* \subset \ran A^*$, i.e., $\sD(A,B)=\sK$.

\medskip

The following simple result is immediate  from Lemma~\ref{TAB}.

\begin{lemma}\label{ABTT}
 Let $A\in\bB(\sE,\sH)$ and $B\in\bB(\sE,\sK)$ and
let the relation $L(A,B)$ be defined by \eqref{Tabegin}.
Then the following statements are equivalent:
\begin{enumerate}\def\labelenumi{\rm(\roman{enumi})}
\item $B$ is dominated by $A$;
\item $\sD(A,B)=\sK$;
\item $L(A,B)$ is a bounded operator.
\end{enumerate}
\end{lemma}

The notion of domination in Definition \ref{D} is now extended.

\begin{definition}\label{AD}
Let $\sE$, $\sH$, and $\sK$ be Hilbert spaces and let $A\in\bB(\sE,\sH)$
and $B\in\bB(\sE,\sK)$.  Then the operator $B$ is said to be
\textit{almost dominated} by $A$,
if there exists a sequence of bounded operators
$B_n \in \bB(\sE,\sK_n)$, where $\sK_n$ are Hilbert spaces,
and a sequence $c_n \geq 0$,
such that for all $f \in \sE$:
\begin{enumerate}\def\labelenumi{\rm(\alph{enumi})}
\item $\|B_{n}f\| \leq c_{n} \|Af\|$;
\item $\|B_{n}f\| \leq \|B_{n+1} f\|$;
\item $\|B_{n}f\| \nearrow \|Bf\|$.
\end{enumerate}
\end{definition}

It is clear  that if $B\in\bB(\sE,\sK)$ is dominated by $A\in\bB(\sE,\sH)$, then  $B$
is automatically almost dominated by $A$ by taking $B_n=B$ and $c_n=c$.

\vspace{0.2cm}

The analog of Lemma \ref{ABTT} for the almost dominated case
is contained in the following theorem.

\begin{theorem}\label{TAB++}
 Let $A\in\bB(\sE,\sH)$ and $B\in\bB(\sE,\sK)$ and
let the relation $L(A,B)$ be defined by \eqref{Tabegin}.
Then the following statements are equivalent:
\begin{enumerate}\def\labelenumi{\rm(\roman{enumi})}
\item $B$ is almost dominated by $A$;
\item $\sD(A,B)$
is dense in $\sK$;
\item $L(A,B)$ is regular.
\end{enumerate}
\end{theorem}

\begin{proof}
(i) $\Rightarrow$ (iii)
Assume that $B$ is almost dominated by $A$.
Then there exists a sequence of operators
$B_n \in \bB(\sE,\sK_n)$
as in Definition \ref{AD}.
Note that if $f\in\ker A$ then $B_nf=0$ due to (a),
and hence $\|Bf\|=\sup \|B_nf\|=0$ due to (c).
One concludes that $\ker A\subset \ker B$,
so that $L(A,B)$ is an operator, see \eqref{Tabb}.
Define the sequence of linear relations $T_n$
from $\cran A$ to $\sK_{n}$ by
\[
 T_{n} =\clos \{\,\{Af,B_nf\}:\, f\in \sE\,\}.
\]
Due to (a) it follows
that each
$T_{n}$ is a closed bounded operator from $\cran A$ to $\sK_{n}$.
Furthermore, by (b) one sees that for $m \leq n$
\[
 \|T_{m} Af\|=\|B_{m}f\| \leq \|B_{n} f\|= \|T_{n} Af\|,
 \quad f \in \sE,
\]
which implies that
\begin{equation}\label{kolme}
 \|T_{m} h \| \leq \|T_{n} h \|, \quad h \in \cran A.
\end{equation}
Moreover, if $h\in\dom L(A,B)$ so that $h=Af$ for some $f\in\sE$,
then it follows from (c) that
\begin{equation}\label{op00}
 \|T_{n}h\|=\|T_{n}Af\|=\|B_nf\| \nearrow \|Bf\|=\|T_{A,B}h\|.
\end{equation}
Hence the sequence $T_{n} \in \bB(\cran A,\sK_{n})$
satisfies \eqref{kolme} and \eqref{op00}.
Thus \cite[Theorem 8.8]{HSS2018} implies that the operator
$L(A,B)$ is closable.

(ii) $\Rightarrow$ (i) Assume that $L(A,B)$ is regular,
so that $L(A,B)$ is a closable operator from $\sH$ to $\sK$.
Then by \cite[Theorem 8.9]{HSS2018}  there exists a sequence
$T_n \in \bB(\cran A,\sH)$ of bounded operators
with the property \eqref{kolme},  such that
\begin{equation}\label{vasa}
 \|T_n h\|\nearrow \|T_{A,B}h\|, \quad h\in\dom L(A,B).
\end{equation}
Define the operators $B_n=T_nA\in\bB(\sE,\sH)$.
It will be shown that the conditions of Definition \ref{AD}
are satisfied with $\sK_{n}=\sH$.
First note that
\[
\|B_n f\|\le \|T_n\|\,\|Af\|, \quad f\in \sE,
\]
so that (a) is satisfied. Secondly, observe that
for all $m\le n$ it follows from \eqref{kolme} that
\[
\|B_m f\|=\|T_m Af\|\le \|T_nAf\|=\|B_n f\|, \quad f \in \sE,
\]
so that (b) is satisfied. Finally note that \eqref{vasa} implies that
\[
 \|B_n f\|=\|T_nAf\| \nearrow \|L(A,B)Af\|=\|Bf\|, \quad f \in \sE,
\]
so that (c) is satisfied.
Thus, $B$ is almost dominated by $A$.

(ii) $\Leftrightarrow$ (iii) See Lemma \ref{TAB}.
\end{proof}

The following definition finds its inspiration in a similar notion
which is current in measure theory.

\begin{definition}\label{S}
Let $\sE$, $\sH$, and $\sK$ be Hilbert spaces and let $A\in\bB(\sE,\sH)$
and $B\in\bB(\sE,\sK)$.
Then the operator $B$ is said to be \textit{singular}
with respect to $A$
or, equivalently, the operator $A$ is said to be
\textit{singular} with respect to $B$,
if for every $ D\in\bB(\sE)$
\[
 D \prec A \quad \text{and} \quad D \prec B
\quad \Rightarrow\quad D=0.
\]
\end{definition}

Note that an equivalent statement is that
$\ran D^* \subset \ran A^*$ and $\ran D^* \subset \ran B^*$ imply that $D=0$.
It is straightforward to characterize the property
"$B$ is singular with respect to $A$"
in terms of the operators $A$ and $B$.

\begin{theorem}\label{ABlemma}
Let $A\in\bB(\sE,\sH)$ and $B\in\bB(\sE,\sK)$ and
let the relation $L(A,B)$ be defined by \eqref{Tabegin}.
Then the following statements are equivalent:
\begin{enumerate}\def\labelenumi{\rm(\roman{enumi})}
\item $B$ is singular with respect to $A$;
\item  $\ran A^* \cap \ran B^*=\{0\}$;
\item $L(A,B)$ is singular.
\end{enumerate}
\end{theorem}

\begin{proof}
 (i) $\Rightarrow$ (ii)
Assume that $B$ is singular with respect to $A$.  To prove (ii),
 suppose that    $\ran A^*\cap\ran B^* \neq \{0\}$.
Then there exists a proper orthogonal projection $D$ in $\sE$
with $\ran D \subset \ran A^*\cap\ran B^*$, or
\[
 \ran D \subset \ran A^* \quad \mbox{and} \qquad  \ran D \subset \ran B^*.
\]
Since $D$ is an orthogonal projection, it is selfadjoint and one concludes
that $D \prec A$ and $D \prec B$. Hence $D=0$.
This contradiction implies that $\ran A^* \cap \ran B^*=\{0\}$.

(ii) $\Rightarrow$ (i) Assume that
 $\ran A^*\cap\ran B^*=\{0\}$.
 To prove (i), suppose that $D\in\bB(\sE)$ satisfies  $D \prec A$ and $D \prec B$
or, equivalently,
 $\ran D^{*} \subset \ran A^{*}$ and
$\ran D^{*} \subset \ran B^{*}$. This leads to
$\ran D^{*} \subset \ran A^{*}\cap\, \ran B^{*}$.
Hence $D^{*}=0$ and thus $D=0$.
Therefore,  $B$ is singular with respect to $A$.

(ii) $\Leftrightarrow$ (iii) See Lemma \ref{labsing}.
\end{proof}

\section{Almost domination and the Radon-Nikodym derivative}\label{sec8}

Let $\sE$, $\sH$, and $\sK$ be Hilbert spaces and let $A\in\bB(\sE,\sH)$
and $B\in\bB(\sE,\sK)$.
Let $L(A,B)$ be the corresponding operator range relation defined in \eqref{Tabegin}.
If $B$ is dominated or almost dominated by $A$,
then there is a factorization of $B$ with respect to $A$, which gives
the notion of the Radon-Nikodym derivative
in the abstract setting of the operator range relation $L(A,B)$.

\vspace{0.2cm}

Observe the following straightforward remarks.
Let $L(A,B)$ be an operator range relation as in \eqref{Tabegin} and
recall that $L(A,B)$ is equal to the  quotient $BA^{-1}$.
In the case that $L(A,B)$ is an operator, one may write
\begin{equation}\label{unno}
 Bf=L(A,B) Af \quad \mbox{for all} \quad f \in \sE.
\end{equation}
If also $L(A,B)^{**}$ is an operator, then  it follows from
\[
 L(A,B)=\{ \{Af,Bf \} : \, f \in \sE \} \subset L(A,B)^{**},
\]
that one may write
\begin{equation}\label{dduo}
 Bf=L(A,B)^{**} Af \quad \mbox{for all} \quad f \in \sE.
\end{equation}
 Note that in this identity only the reduced part of the pair $A$ and $B$ plays a role.
\medskip

First the case of domination will be considered.

\begin{lemma}\label{thmRNder0}
Let $A\in\bB(\sE,\sH)$ and $B\in\bB(\sE,\sK)$.
Then the following statements are equivalent:
\begin{enumerate}\def\labelenumi{\rm(\roman{enumi})}
\item $B$ is dominated by $A$;
\item $B=CA$ holds for some bounded linear operator $C$ from $\sH$ to $\sK$;
\item $B=CA$ holds for some closed bounded linear operator $C$ from $\sH$ to $\sK$.
\end{enumerate}
If one of these conditions is satisfied, then $L(A,B)$ is a bounded linear operator
and $B=L(A,B)A$. Moreover,
if $B=CA$ holds for some bounded linear operator $C$ from $\sH$ to $\sK$,
then $L(A,B) \subset C$; and if $B=CA$ holds for some closed bounded linear operator
$C$ from $\sH$ to $\sK$, then $L(A,B)^{**} \subset C$.
 \end{lemma}

\begin{proof}
(i) $\Rightarrow$ (ii) It follows from Lemma  \ref{ABTT} that $L(A,B)$
is a bounded linear operator. Hence it follows from \eqref{unno} that (ii) holds.

(ii) $\Rightarrow$ (iii) To see this replace $C$ in $B=CA$ by its closure $C^{**}$.

(iii $\Rightarrow$ (i) This is clear.

 These equivalent statements imply that
$L(A,B)$ is a bounded linear operator and
it follows from \eqref{Tabegin} that $B=L(A,B)A$.
Therefore, if $B=CA$ holds for some bounded linear operator
$C$ from $\sH$ to $\sK$, then $L(A,B) \subset C$ and, if in addition $C$ is closed,
it follows that $L(A,B)^{**} \subset C$.
\end{proof}

Notice that the Radon-Nikodym derivative in the following definition
satisfies the minimality property expressed in Lemma \ref{thmRNder0}.

\begin{definition}\label{rrnn}
Let $A\in\bB(\sE,\sH)$ and  $B\in\bB(\sE,\sK)$ and
assume that $B$ is  dominated by $A$.
Then the \textit{Radon-Nikodym derivative} $R(A,B)$
of $B$ with respect to $A$ is the bounded closed
operator $L(A,B)^{**}$ from $\sH$ to $\sK$.
\end{definition}

For an illustration of such a Radon-Nikodym derivative,
return  to the inequalities $A^*A \leq A^*A+B^*B$ and $B^*B \leq A^*A+B^*B$.
 These inequalities  imply that
there are $C_A \in \bB( \sE, \sH)$ and $C_B \in \bB( \sE, \sK)$,
such that  the identities in \eqref{eqCACB++}
hold, and they are unique when  \eqref{eqCACBa} is assumed.
It is clear from Definition \ref{D} that $A$ and $B$ are dominated by
the operator $(A^*A+B^*B)^\half$, so that each of the
following relations from $\sE$ to $\sH$ and from $\sE$ to $\sK$, respectively,
\[
 L((A^*A+B^*B)^{\half}, A) \quad \mbox{and} \quad L((A^*A+B^*B)^{\half}, B)
 \]
is not only regular, but (the graph of) a bounded operator.
Recall that $A_0$ and $B_0$ are the reduction of the pair $A$ and $B$;
see \eqref{reduu0} and  Lemma \ref{rreduxx}.

\begin{lemma}\label{nneeww}
Let $A\in\bB(\sE,\sH)$ and $B\in\bB(\sE,\sK)$,
and let $C_{A}$ and $C_{B}$ be the uniquely defined operators
satisfying \eqref{eqCACB++} and \eqref{eqCACBa}.
 Then the Radon-Nikodym derivatives of $A$ and $B$
 with respect to $(A^*A+B^*B)^{\half}$ are given by
\begin{equation}\label{RNCACB}
\begin{array}{c}
 R((A^*A+B^*B)^{\half}, A)= C_{A_0}, \quad
  R((A^*A+B^*B)^{\half}, B)= C_{B_0},
  \end{array}
\end{equation}
\end{lemma}

\begin{proof}
 Since the operators $C_A$ and $C_B$ satisfy the identities \eqref{eqCACB++},
 if follows from Lemma \ref{thmRNder0} that
\[
 R((A^*A+B^*B)^{\half}, A)^{**} \subset C_A, \quad
 R((A^*A+B^*B)^{\half}, B)^{**} \subset C_B.
\]
Since all these operators are closed and bounded,
\[
 \dom R((A^*A+B^*B)^{\half}, A)^{**}=\dom R((A^*A+B^*B)^{\half}, B)^{**}=\cran (A^*A+B^*B)
\]
and the identities \eqref{RNCACB} follow.
\end{proof}

Next, the notion of almost domination in the general case will be taken up again.

\begin{theorem}\label{thmRNder}
Let $A\in\bB(\sE,\sH)$ and $B\in\bB(\sE,\sK)$.
Then the following statements are equivalent:
\begin{enumerate}\def\labelenumi{\rm(\roman{enumi})}
\item $B$ is almost dominated by $A$;
\item $B=CA$ holds for some closable linear operator $C$ from $\sH$ to $\sK$;
\item $B=CA$ holds for some closed linear operator $C$ from $\sH$ to $\sK$.
\end{enumerate}
If one of these conditions is satisfied, then $L(A,B)$ is a closable linear operator
such that $B=L(A,B)A$. Moreover,
if $B=CA$ holds for some closable linear operator $C$ from $\sH$ to $\sK$,
then $L(A,B) \subset C$ and if $B=CA$ holds for some closed linear operator
$C$ from $\sH$ to $\sK$, then $L(A,B)^{**} \subset C$.
\end{theorem}

\begin{proof}
(i) $\Rightarrow$ (ii)
Assume that $B$ is almost dominated by $A$.
By Theorem \ref{TAB++} this implies that
$L(A,B)$ is a closable operator.
In particular, it follows from \eqref{dduo} that (ii) holds.

(ii) $\Rightarrow$ (iii)   As in the previous lemma this is
seen by replacing $C$ in $B=CA$ by its closure.

(iii) $\Rightarrow$ (i)  Assume that $B=CA$ with a closed operator $C$.
It follows from $B=CA$ that $A^{*}C^{*} \subset B^{*}$, in other words
\[
\{k,k'\} \in C^{*}  \quad \Rightarrow \quad  \{k, A^{*} k'\} \in B^{*}
\quad \Rightarrow \quad B^{*}k=A^{*}k'.
\]
Hence $\dom C^{*} \subset \sD(A,B)$. Since $C$ is a closed operator,
it follows that $\dom C^{*}$ is dense and thus that $\sD(A,B)$ is dense.
By Theorem \ref{TAB++} this means that $B$
is almost dominated by $A$.

It remains to prove the last statements. Notice that if $C$ satisfies (ii), then
 $\ran A\subset \dom C$ and
\[
L(A,B)=\{\{Ah,Bh\}:\, h \in \sE\} = \{\{Ah,CAh\}:\, h \in \sE\} \subset C.
\]
If, in addition, $C$ is closed, then one sees that  $L(A,B)^{**}\subset C$.
\end{proof}

Similar to what happens in the dominated case,
the Radon-Nikodym derivative in the following definition satisfies
the minimality property expressed in Theorem \ref{thmRNder}.

\begin{definition}\label{rrnn+}
Let $A\in\bB(\sE,\sH)$ and $B\in\bB(\sE,\sK)$ and
assume that $B$ is almost dominated by $A$.
Then the \textit{Radon-Nikodym derivative} $R(A,B)$
of $B$ with respect to $A$ is the  closed
operator $L(A,B)^{**}$ from $\sH$ to $\sK$.
\end{definition}

\begin{corollary}\label{thmRNder0a}
Let $A\in\bB(\sE,\sH)$ and $B\in\bB(\sE,\sK)$ and
assume that $B$ is almost dominated by $A$.
Then the \textit{Radon-Nikodym derivative} $R(A,B)$
of $B$ with respect to $A$ is  bounded if and only if
$B$ is dominated by $A$.
\end{corollary}

Now recall that $L(A,B)^{**}=L(C_A,C_B)$ and, moreover, that
$L(C_A,C_B)$ is an operator precisely when $\ker C_A \subset \ker C_B$.
The Radon-Nikodym derivative $R(A,B)$
can be expressed
in terms of the Radon-Nikodym derivatives  in Lemma \ref{nneeww}.

\begin{theorem}\label{rrnn1}
Let $A\in\bB(\sE,\sH)$ and $B\in\bB(\sE,\sK)$ and
assume that $B$ is dominated or almost dominated by $A$.
Moreover, let $C_{A}$ and $C_{B}$ be as defined
in \eqref{eqCACB++} and \eqref{eqCACBa} and let
$C_{A_0}$ and $C_{B_0}$ be the Radon-Nikodym derivatives in \eqref{RNCACB}.
Then the Radon-Nikodym derivative $R(A,B)$
of $B$ with respect to $A$ is given by  the quotient
\begin{equation}\label{RNforAB}
   R(A,B)    = C_{B_0}(C_{A_0})^{-1},
 \end{equation}
where  $\ker C_{A_0} \subset \ker C_{B_0}$.
 \end{theorem}

\begin{proof}
 The Radon-Nikodym derivative $R(A,B)$ is given
 by the closed operator $L(C_A,C_B)$.
Now consider the reduction $A_0$ and $B_0$ of the pair $A$ and $B$.
By Lemma \ref{rreduxx} one has
\[
 R(A,B)=(L(A,B)^{**}=L(C_A, C_B)=L(C_{A_0}, C_{B_0}),
\]
with $\ker C_{A_0} \subset \ker C_{B_0}$.
The identity \eqref{RNforAB} by rewriting the above result
as a quotient of the operators $C_{A_0}$ and $C_{B_0}$.
 \end{proof}

 It will be helpful to compare the results in Lemma \ref{nneeww}
 and Theorem \ref{rrnn1}, together
with Theorem \ref{Main},  with the following construction known
from measure theory.
Let $(\mu,\nu)$ be a pair of finite positive measures.
Then $\mu$ and $\nu$ are absolutely continuous
with respect to the sum measure $\rho=\mu+\nu$: $\mu\ll\rho$, $\nu\ll\rho$.
This gives rise to the existence of the corresponding
Radon-Nikodym derivatives $f=\frac{d\mu}{d\rho}$
and $g=\frac{d\nu}{d\rho}$ and in this case
\begin{equation}\label{RNsum}
 f+g=1 \quad \rho\text{-a.e.}
\end{equation}
If, in addition, $\nu \ll \mu$ with the Radon-Nikodym derivative $h=\frac{d\nu}{d\mu}$,
then $\nu\ll\mu\ll\rho$ implies that
\[
 g=\frac{d\nu}{d\rho}=\frac{d\nu}{d\mu} \, \frac{d\mu}{d\rho} = hf \quad \rho\text{-a.e.}
\]
Since $\rho=\mu+\nu$ and $\nu \ll \mu$, one has also $\rho \ll \mu$.
Consequently, one has $f>0$ $\rho$-a.e. ($\Leftrightarrow$ $\mu$-a.e.) and
thus, in fact, the Radon-Nikodym derivative $h$ is given by
\begin{equation}\label{RNfrac}
 \frac{d\nu}{d\mu}=\frac{g}{f} \quad \mu\text{-a.e.}
\end{equation}

\begin{remark}
 The concept of Radon-Nikodym derivative given in Definitions~\ref{rrnn} and~\ref{rrnn+}
is applicable and uniquely determined for general operator range relations,
which are regular (i.e. closable operators).
Indeed, by Theorem \ref{Tclosed} the operator $T$ has the representation $T=L(A,B)$ and
if $T=\ran Z$ for some other operator $Z\in\bB(\sY,\sH\times\sK)$,
then Lemma \ref{opran+++} shows that there exists a bounded and boundedly invertible
operator $W\in\bB(\sE,\sY)$ such that $c(A,B)=ZW$.
Then
\[
 L(A,B)=\ran c(A,B)=\ran ZW=\ran Z
\]
and taking closures leads to $\cran Z=R(A,B)$.
\end{remark}

\section{Lebesgue type decompositions for pairs of bounded linear operators}\label{sec9}

Let $A\in\bB(\sE,\sH)$ and $B\in\bB(\sE,\sK)$ be bounded linear operators.
In this section it will be shown that there exist Lebesgue type decompositions
$B=B_1+B_2$ such that $B_1$ is almost dominated by $A$
and $B_2$ is singular with respect to $A$. The main idea is to go back to the
corresponding operator range relation $L(A,B)$ and to use the
Lebesgue type decompositions of $L(A,B)$; cf. \cite{HSS2018}.

\begin{definition}\label{bregsing}
Let $A\in\bB(\sE,\sH)$ and
$B\in\bB(\sE,\sK)$ be bounded linear operators
and let $P$ be the orthogonal projection onto $\sD(A,B)^{\perp}$.
The regular part $B_{\rm {reg}}$ and the singular part $B_{\rm {sing}}$
are defined by
\begin{equation}\label{aregsing}
B_{\rm {reg}}=(I-P)B, \quad B_{\rm {sing}}=PB.
\end{equation}
The corresponding  decomposition
\begin{equation}\label{aregsing1}
 B=B_{\rm {reg}}+B_{\rm {sing}}.
\end{equation}
is called the Lebesgue decomposition of $B$ with respect to $A$.
\end{definition}

Let $A\in\bB(\sE,\sH)$ and $B\in\bB(\sE,\sK)$ be
as in Definition \ref{bregsing}.
Let $L(A,B)$ from $\sH$ to $\sK$ be defined as in  \eqref{Tabegin}
and recall that
\[
\dom L(A,B)^*=\sD(A,B) \quad \mbox{and} \quad \mul L(A,B)^{**}=\sD(A,B)^{\perp};
\]
cf. \eqref{tabst1}.
Then
the relation $L(A,B)$ has the Lebesgue decomposition
\begin{equation}\label{Tabdec}
 L(A,B)=L(A,B)_{\rm{{reg}}}+L(A,B)_{\rm{{sing}}},
\end{equation}
where the regular and singular components are given by
\begin{equation}\label{Tabdecc}
  L(A,B)_{\rm{{reg}}}=(I-P)L(A,B), \quad L(A,B)_{\rm{{sing}}}=PL(A,B);
\end{equation}
here $P$ stands for the orthogonal projection from $\sK$
onto  $\mul L(A,B)^{**}=\sD(A,B)^\perp$;
cf. \cite{HSS2018}.
Via the decomposition \eqref{Tabdec}
one may now obtain the Lebesgue decomposition
of $B$ with respect to $A$ as in Definition \ref{bregsing}.

\begin{theorem}[Lebesgue decomposition]\label{LebABthm}
Let $A\in\bB(\sE,\sH)$ and
$B\in\bB(\sE,\sK)$ be bounded linear operators.
Then $B_{\rm {reg}}$ is almost dominated by $A$,
$B_{\rm {sing}}$ is singular with respect to $A$, and
$B$ has the Lebesgue decomposition \eqref{aregsing1}
with respect to $A$.
 The regular part $B_{\rm reg}$ can be written as
\begin{equation}\label{aregsing2}
 B_{\rm reg}=R(A,B_{\rm reg})A,
\end{equation}
where $R(A,B_{\rm reg})$ is the Radon-Nikodym derivative of
$B_{\rm {reg}}$ with respect to $A$:
\begin{equation}\label{aregsing3}
 R(A,B_{\rm reg})=L(A,B_{\rm reg})^{**}.
\end{equation}
In fact, if $L(A,B)$ is given by \eqref{Tabdec} and \eqref{Tabdecc}, then
\begin{equation}\label{Tab11}
 L(A,B)_{\rm {reg}}=L(A, B_{\rm reg}), \quad L(A,B)_{\rm {sing}}=L(A, B_{\rm sing}).
 \end{equation}
\end{theorem}

\begin{proof}
The decomposition \eqref{aregsing1} follows from \eqref{aregsing}.
 It follows from \eqref{Tabdecc}
that $L(A,B)_{\rm{{reg}}}$ and  $L(A,B)_{\rm{{sing}}}$
have the representations
\[
\begin{split}
 &L(A,B)_{\rm {reg}}=\{\,\{Af,(I-P)Bf\}:\, f\in\sE\,\}=\{\,\{Af,B_{\rm {reg}}f\}:\, f\in\sE\,\}, \\
  &L(A,B)_{\rm {sing}}=\{\,\{Af,P B f\}:\, f\in\sE\,\}=\{\,\{Af, B_{\rm {sing}}f\}:\, f\in\sE\,\},
\end{split}
\]
which give \eqref{Tab11}.
Since the relation $L(A,B)_{\rm {reg}}$ is regular,
it follows from Theorem \ref{TAB++}
that $B_{\rm {reg}}$ is almost dominated by $A$.
Likewise,
since the relation $L(A,B)_{\rm {sing}}$ is singular,
it follows from Theorem \ref{ABlemma}
that $B_{\rm {sing}}$ is singular with respect to $A$.

The statements about the Radon-Nikodym derivative in  \eqref{aregsing2}
and \eqref{aregsing3} follow from Theorem \ref{thmRNder}.
\end{proof}

The Lebesgue decomposition in \eqref{aregsing1}
is an example of a so-called Lebesgue type decomposition.

\begin{definition}\label{optype}
Let $A\in\bB(\sE,\sH)$ and $B\in\bB(\sE,\sK)$ be bounded operators.
The operator $B$ is said to have a Lebesgue type
decomposition with respect to $A$,
if $B=B_{1}+B_{2}$ where $B_{1}, B_{2} \in\bB(\sE,\sK)$ have the properties:
\begin{enumerate}\def\labelenumi{\rm(\alph{enumi})}
\item $\ran B_{1} \perp \ran B_{2}$;
\item $B_{1}$ is almost dominated by $A$;
\item $B_{2}$ is singular with respect to $A$.
\end{enumerate}
\end{definition}

Let $A\in\bB(\sE,\sH)$ and $B\in\bB(\sE,\sK)$
as in Definition \ref{optype}.
Let the linear relation $L(A,B)$  from $\sH$ to $\sK$ be defined by \eqref{Tabegin}.
 According to \cite{HSS2018}  the Lebesgue type decompositions of $L(A,B)$
are in one-to-one correspondence
 with the closed linear subspaces $\sL \subset \sK$ such that
\begin{equation}\label{nic2}
\sL \subset \cdom L(A,B)^* \setminus \dom L(A,B)^*,
\end{equation}
which satisfy the condition
\begin{equation}\label{Lleb2+}
  \clos( \sL^\perp \cap \sD(A,B)) = \sL^\perp \cap \clos \sD(A,B).
\end{equation}
Define the closed linear subspace $\sM$ by
\begin{equation}\label{mmm}
\sM=\sD(A,B)^\perp \oplus \sL,
\end{equation}
and let $P_\sM$ be the orthogonal projection from $\sK$ onto $\sM$.
Then the corresponding  Lebesgue type decomposition of $L(A,B)$ is given by
\begin{equation}\label{nic5}
 L(A,B)=L(A,B)_1+L(A,B)_2,
\end{equation}
where the regular and singular components are given by
\begin{equation}\label{nic3}
 L(A,B)_1=(I-P_\sM) L(A,B), \quad L(A,B)_2= P_\sM L(A,B),
\end{equation}
respectively.
Via the decomposition \eqref{nic5} one may now obtain the
Lebesgue type decompositions of $B$ with respect to $A$ as in Definition \ref{optype}.

\begin{theorem}\label{llaabb}
Let $A\in\bB(\sE,\sH)$ and
$B\in\bB(\sE,\sK)$ be bounded linear operators.
Then the Lebesgue type decompositions of $B$
with respect to $A$ are in one-to-one correspondence
 with the closed linear subspaces $\sL\subset \sK$ in \eqref{nic2}
 which satisfy the condition \eqref{Lleb2+}.
 In particular, the corresponding Lebesgue type decomposition is given by
\[
 B=B_{1}+B_{2}, \quad B_{1}=(I-P_\sM)B, \quad
 B_{2}=P_\sM B,
\]
where $\sM$ is as in \eqref{mmm} and $P_\sM$
is the orthogonal projection from $\sK$ onto $\sM$, while
$B_{1}$ is almost dominated by $A$ and
$B_{2}$ is singular with respect to $A$.
 The regular part $B_{1}$ can be written as
\[
 B_{1}=R(A,B_1)A,
\]
where $R(A, B_1)$ is the Radon-Nikodym derivative of
$B_{1}$ with respect to $A$:
\[
 R(A,B_1)=L(A,B_1)^{**}.
\]
In fact, $L(A,B)$ is given by \eqref{nic5} and \eqref{nic3} precisely, when
\[
 L(A,B)_{\rm {1}}=L(A, B_{1}), \quad L(A,B)_{2}=L(A, B_{2}).
\]
\end{theorem}

\begin{proof}
 Consider the linear relation $L(A,B)$ from $\sH$ to $\sK$ defined by \eqref{Tabegin}.

First it is shown that every Lebesgue type decomposition
of the linear relation $L(A,B)$ (in the sense of \cite{HSS2018})
generates a Lebesgue type decomposition of $B$
with respect to $A$ as in Definition \ref{optype}.
To see this let  $\sL \subset \cdom L(A,B)^* \setminus \dom L(A,B)$
be a linear subspace
which satisfies  \eqref{Lleb2+} and let $\sM$ be as defined
in \eqref{mmm} with the corresponding
orthogonal projection $P_\sM$ onto $\sM$.
According to \cite[Theorem~5.4]{HSS2018} the formula
\[
 L(A,B)=(I-P_\sM) L(A,B) + P_\sM L(A,B)
\]
determines a Lebesgue type decomposition of $L(A,B)$,
where $(I-P_\sM) L(A,B)$ is the regular part and
$P_\sM L(A,B)$ is the singular part generated uniquely by the subspace $\sL$.
From the representation of the regular part
\[
 (I-P_\sM) L(A,B)=\{\,\{Af, (I-P_\sM) Bf\}:\, f\in\sE\,\}
\]
and Theorem \ref{TAB++} it follows that $B_1=(I-P_\sM) B$
is almost dominated by $A$.
Likewise from the representation of the singular part
\[
 P_\sM L(A,B)=\{\,\{Af,  P_\sM Bf\}:\, f\in\sE\,\}
\]
and Theorem \ref{ABlemma} it follows that  $B_2=P_\sM B$
is singular with respect to $A$.
 Hence, the identity $B=B_1+B_2$ is a Lebesgue type decomposition of $B$ with respect to $A$
in the sense of Definition \ref{optype}.

Conversely, assume that  $A\in\bB(\sE,\sH)$, $B\in\bB(\sE,\sK)$,
and that $B$ has a Lebesgue type decomposition as in Definition \ref{optype}.
Then it is clear that the corresponding relations satisfy
 \begin{equation}\label{vaasnew}
L(A,B)=L(A,B_{1})+L(A,B_{2}),
\end{equation}
where, due to Theorem \ref{TAB++} and Theorem \ref{ABlemma},
the relation $L(A,B_{1})$ is regular and the relation $L(A,B_{2})$ is singular.
Hence \eqref{vaasnew} is a Lebesgue type decomposition for $L(A,B)$.
Again by \cite[Theorem~5.4]{HSS2018} there exists a linear subspace $\sL$, such that
\eqref{nic2} and \eqref{Lleb2+} are satisfied, and a subspace $\sM$ given by \eqref{mmm}
such that
 \begin{equation}\label{grijps1}
 \begin{split}
 L(A,B_{1})&=(I-P_{\sM}) L(A,B)=L(A, (I-P_{\sM})B),  \\
  L(A,B_{2})&=P_{\sM} L(A,B)=L(A, P_{\sM}B).
\end{split}
\end{equation}
Thanks to the first identities in \eqref{grijps1}, for every $h \in \sE$
there exists $f \in \sE$, such that
\[
 Ah=Af, \quad (I-P_{\sM})Bh=B_{1}f.
\]
Thus $f=h+\varphi$ for some $\varphi \in \ker A$. Since $L(A,B_{1})$ is regular,
it is an operator and hence $\mul L(A,B_1)=B_1(\ker A)=0$; cf.  \eqref{Tabb}.
This shows that $B_{1}\varphi=0$, and thus $(I-P_{\sM})Bh=B_{1}h$.
Therefore, $(I-P_{\sM})B=B_{1}$ and, consequently, $P_{\sM}B=B_{2}$.
 This proves the one-to-one correspondence between the Lebesgue type decompositions
of $B=B_1+B_2$ in Definition \eqref{optype} and of $L(A,B)$ in \eqref{nic5} and \eqref{nic3}.
The one-to-one correspondence between the closed subspaces $\sL$ satisfying the conditions
\eqref{nic2} and \eqref{Lleb2+} is obtained from \cite[Theorem~5.4]{HSS2018}.

The statement about the Radon-Nikodym derivative of $B_{1}$ with respect to $A$
follows from Theorem \ref{TAB++}.
\end{proof}

For any Lebesgue type decomposition of $B$ with respect to $A$, the part $B_1$,
which is almost dominated by $A$, is, in fact, dominated by the regular part $B_{\rm reg}$.

\begin{corollary}\label{optim}
Let $A\in\bB(\sE,\sH)$ and $B\in\bB(\sE,\sK)$.
Let $B=B_1+B_2$ be a Lebesgue type decomposition of $B$, then
\[
 \|B_1 h\| \leq \|B_{\rm reg} h\|, \quad h \in \sE.
\]
\end{corollary}

\begin{proof}
 Let $B=B_1+B_2$ be a Lebesgue type decomposition of $B$
with respect to $A$. Then as in the proof of Theorem \ref{llaabb}
one finds that
\[
 B_1=(I-P_{\sM})B=(I-P_\sM)(I-P)B=(I-P_\sM) B_{\rm reg},
\]
where it was used that   $\ran P \subset \ran P_{\sM}$; cf. \eqref{mmm}.
 \end{proof}

The uniqueness of Lebesgue type decompositions of $B$
with respect to $A$ in Definition \ref{optype}
can be characterized as follows.

\begin{corollary}\label{Andoun-}
Let $A\in\bB(\sE,\sH)$
and $B\in\bB(\sE,\sK)$.
Then the following statements are equivalent:
\begin{enumerate}\def\labelenumi{\rm(\roman{enumi})}
\item $B$ admits a unique Lebesgue type decomposition with respect to $A$;
\item $L(A,B)$ admits a unique Lebesgue type decomposition;
\item $\sD(A,B)$ is closed;
\item $B_{\rm reg}$ is dominated by $A$;
 \item the Radon-Nikodym derivative $R(A,B_{\rm reg})$  is a bounded operator.
\end{enumerate}
In this case, all Lebesgue type decompositions of $B$ with respect to $A$
coincide with the Lebesgue decomposition of
$B=B_{\rm reg}+B_{\rm sing}$ in Theorem \ref{LebABthm}.
\end{corollary}

\begin{proof}
(i) $\Leftrightarrow$ (ii)
The Lebesgue type decompositions of $B=B_{1}+B_{2}$ with respect to $A$
correspond to the Lebesgue type decompositions
of $L(A,B)=T_{1}+T_{2}$ via Theorem \ref{llaabb}.
Hence $B$ has a unique Lebesgue type decomposition with respect to $A$
if and only $L(A,B)$ has a unique Lebesgue type decomposition.

(ii) $\Leftrightarrow$ (iii)
By \cite[Theorem~6.1]{HSS2018} $L(A,B)$ has a unique Lebesgue type decomposition
if and only if $\dom L(A,B)^*$ is closed, i.e., $\sD(A,B)$ is closed; cf. \eqref{tabst}.

(ii) $\Leftrightarrow$ (iv) Again by \cite[Theorem~6.1]{HSS2018}
$L(A,B)$ has a unique Lebesgue type decomposition if and only if
$L(A,B)_{\rm {reg}}=L(A, B_{\rm reg})$ is bounded; cf. Theorem \ref{LebABthm}.
Now $L(A, B_{\rm reg})$ is bounded if and only if
$B_{\rm reg}$ is dominated by $A$; cf Lemma \ref{ABTT}.

(iv) $\Leftrightarrow$ (v) This follows from Corollary \ref{thmRNder0a}.

The last statement is clear from \eqref{nic2}, since if $\sD(A,B)=\dom L(A,B)^*$ is closed,
then $\sL=\{0\}$ and $\sM$ in \eqref{mmm} coincides with $\sD(A,B)^\perp$.
 \end{proof}

\end{document}